\documentclass[reqno,11pt]{amsart}

\usepackage{amsmath,amssymb,amsthm,mathrsfs,graphicx,enumitem,amsfonts}
\usepackage{epsfig,color,mathtools}

\usepackage[numbers]{natbib}

 \allowdisplaybreaks
\numberwithin{equation}{section}
\mathtoolsset{showonlyrefs}

\def\H{\mathcal H}

\def\R{\mathbb R}
\def\N{\mathbb N}

\def\dist{\hbox{dist}}
\def\e{\varepsilon}
\def\s{\sigma}

\def\vphi{\varphi}

\def\om{\omega}
\def\l{\lambda}
\def\g{\gamma}

\def\Om{\Omega}
\def\de{\delta}

\def\p{\mathbf{p}}
\def\na{\nabla}

\def\pa{\partial}

\def\E{\mathcal{E}}

\def\bd{{\rm bd}\,}

\renewcommand{\a}{\alpha}
\renewcommand{\b}{\beta}

\renewcommand{\l}{\lambda}
\renewcommand{\L}{\Lambda}

\newcommand{\var}{\varphi}

\renewcommand{\t}{\tau}

\newcommand{\ov}{\overline}

\newcommand{\diam}{\mathrm{diam}}

\newcommand{\Span}{\mathrm{span}}

\def\bd{{\rm bd}\,}

\def \HH{\mathbf{H}}

\def\p{\mathbf{p}}

\theoremstyle{plain}
\newtheorem{theorem}{Theorem}[section]
\newtheorem{lemma}[theorem]{Lemma}
\newtheorem{corollary}[theorem]{Corollary}
\newtheorem{proposition}[theorem]{Proposition}

\newtheorem{remark}[theorem]{Remark}
\newenvironment{MYitemize}{%
\begin{itemize}}{\end{itemize}}

\title{Axial symmetry for fractional capillarity droplets}
\author{C. Mihaila}
\begin{document}
\begin{abstract}
{\rm } A classical result of Wente, motivated by the study of sessile capillarity droplets, demonstrates the axial symmetry of every hypersurface which meets a hyperplane at a constant angle and has mean curvature dependent only on the distance from that hyperplane \cite{Wente80}. An analogous result is proven here for the fractional mean curvature operator.
\end{abstract}
\maketitle

\section{Introduction} Our motivation is the celebrated result by Wente \cite{Wente80}, which shows that a constant mean curvature hypersurface with constant contact angle along an hyperplane is a spherical cap. The more general statement contained in Wente's paper is actually concerned with hypersurfaces whose mean curvature depends only on the distance from their bounding hyperplane. More precisely, in \cite{Wente80}, it was proven that:

\medskip

\noindent {\it Let $E\subset\R^n$, $n\ge 2$, be a bounded open connected subset of $H=\{x_n>0\}$ such that $M=\ov{H\cap\pa E}$ is a $C^2$-hypersurface with boundary ${\rm bd}(M)=M\cap\pa H$. If, for a suitable constant $\s\in(-1,1)$ and function $g:(0,\infty)\to\R$,
\begin{equation*}\begin{split}
&\HH_E(q)=g(q_n)\qquad\forall q\in M\cap H\,,
\\
&\nu_E(q)\cdot e_n=\s\qquad\forall q\in M\cap\pa H\,,
\end{split}\end{equation*}
then $\{M\cap\{x_n=t\}:t>0\}$ is a family of $(n-2)$-dimensional spheres centered on the same vertical axis. Here $\ov{A}$ and $\pa A$ denote the topological closure and the topological boundary of any $A\subset \R^n$, $\bd(M)$ is the manifold boundary, and $\HH_E$ is the mean curvature of $M$ computed with respect to the outer unit normal $\nu_E$ to $E$.}

\medskip

We want to prove a generalization of this theorem with the classical mean curvature replaced by the fractional mean curvature, which was introduced in \cite{caffaroquesavin} as the first variation of fractional perimeter (see the discussion below Theorem \ref{thm axially symmetric} for additional context). Let us recall that, if $E$ is an open subset of $\R^n$, then the fractional mean curvature of order $s\in(0,1)$ is defined for each $q\in\pa E$ as
\begin{equation}
\label{fractional mean curvature}
\HH^{s}_{E} (q):={\rm p.v.}\,\int_{\R^n}\frac{\chi_{E^c}(x)-\chi_{E}(x)}{|x-q|^{n+s}}dx={\rm p.v.}\,\int_{\R^n}\frac{\tilde{\chi}_{E}(x)}{|x-q|^{n+s}}dx\,,
\end{equation}
where $\chi_E$ is the characteristic function of $E$ and 
\[
\tilde{\chi}_E:=\chi_{E^c}-\chi_E\,.
\]
 The integral in \eqref{fractional mean curvature} converges at $q\in\pa E$ as soon as $E$ is of class $C^{1,\a}$ around $q$ for some $\a\in(s,1)$. Indeed, in this case, given $r>\e>0$, the factor $\tilde{\chi}_{E}$ allows one to localize the integral of $|x-q|^{-n-s}$ over $B_r(q)\setminus B_\e(q)$ to a smaller region of the form $P\setminus B_\e(q)$, where $P\subset B_r(q)$ is a set enclosed between two tangent $C^{1,\a}$-paraboloids. For a set $E$ as in Wente's theorem, the size of the region where this kind of cancellation is possible becomes increasingly small as $q$ approaches $\pa H$, and as a consequence $\HH^s_{E}(q)$ will blow-up as $q_n\to 0^+$, at a rate defined by the contact angle. More precisely, as we show in Proposition \ref{H blow}, if $E$ satisfies the fractional variant of the assumptions of Wente's theorem, then
\begin{equation}\label{blow up}
\HH_E^s(q)=\frac{c(n,s,\s)}{q_n^s}\Big(1+{\rm O}(1)\Big)\qquad\mbox{as $q_n\to 0^+$ with $q\in \pa E\cap H$}\,.
\end{equation}
This singular behavior is an unavoidable and challenging feature of extending Wente's result to the fractional setting; overcoming it will be the most interesting point in the proof of the main result. 

To state our result we define the $s$-deficit of $E$
\[
\de_s(E)=\diam(E)^{s+1}\sup\limits_{p_n=q_n \atop p,q\in M\cap H} \frac{|\HH_E^s(p)-\HH_E^s(q)|}{|p-q|^{}}\,,
\]
to measure how far away $\HH_E^s$ is from being constant on horizontal slices in $H$. Note that in this definition the $\diam(E)^{s+1}$ term is used to enforce that $\de_s(E)$ is scale invariant. The following theorem explains how, for small $\de_s,$ we have that $E$ is almost axially symmetric with respect to $\de_s$: 

\begin{theorem}\label{thm axially symmetric}
Let $s\in(0,1)$, $\a\in(s,1)$, and let $E\subset\R^n$, $n\ge 2$, be a bounded open connected subset of $H=\{x_n>0\}$ such that $M=\ov{\pa E\cap H}$ is a $C^{2,\a}$-hypersurface with boundary ${\rm bd}(M)=M\cap\pa H$. Let  $\s\in(-1,1)$ such that
\begin{equation}\label{sigma}
\begin{split}
&\nu_E(q)\cdot\nu_H(q)=\s\qquad\forall q\in M\cap\pa H\,.
\end{split}
\end{equation}
\begin{enumerate}[label=(\alph*)]
\item For every direction $e\in S^{n-1}\cap\pa H$ there is a hyperplane $\pi_e$ orthogonal to $e$ such that, if we let $\rho$ denote reflection across $\pi_{e}$, then there exists a constant $C=C(n,s)$ dependent only on $n$ and $s$, such that 
\begin{equation*}
|E\Delta\rho (E)|\le C\diam(E)^{n}\sqrt{\de_s(E)}\,.
\end{equation*}
\smallskip
\item\label{improved thm} For $h>0$ such that $E_h=\{x\in E:x_n=h\}\ne \emptyset,$ set 
\[
r_h=\inf\limits_{x\in \bd (E_h)}|x-h e_n|,\qquad R_h=\sup\limits_{x\in \bd( E_h)}|x-h e_n|\,,
\]
and 
\[
D_{h}:=B_{r_h}(he_n)\cap \{x:x_n=h\}\, .
\]
There exist $\de_0=\de_0(n,s)$ and $C_0=C_0(n)$ such that, if $\de_s(E)\le \de_0,$ then
\begin{equation}\label{R diff bound}
\frac{R_h-r_h}{\diam(E)}\le 2C_0\frac{\diam(E)^{n}}{|E|}\sqrt{\de_s(E)} ,\qquad \forall h\in (0,\sup\{x_n:x\in E\})\,,
\end{equation}
up to horizontal translations. 
Moreover, if 
\[
\diam(E_h)> 6C_0\frac{\diam(E)^{n+1}}{|E|}\sqrt{\de_s(E)}\,,
\]
then
\[
D_{h}\subset E_h\,.
\]
or, if not, then
\begin{equation}\label{Rh bound}
\frac{R_h}{\diam(E)}\le 7C_0\frac{\diam(E)^{n}}{|E|}\sqrt{\de_s(E)}\,.
\end{equation}
\end{enumerate}

\end{theorem}
\begin{remark}{\rm The constants in the statement are computable and are included in the proof. }
\end{remark}

\begin{remark} {\rm The second part  of the theorem implies that for a small enough cross-sectional diameter, we do not expect that the cross section is centered around the $x_n$-axis, but rather that it is contained in a small ball around the axis. However, if the cross section has large enough diameter in terms of the deficit, then it is pinched between two balls of close radii, both centered around the $x_n$-axis. }
\end{remark}

\begin{remark}
{\rm The term $\diam(E)^n/|E|$, in \eqref{R diff bound} and \eqref{Rh bound}, is scale invariant and might be unbounded for a sequence of $E_m$ satisfying the hypotheses of the theorem. Indeed we could consider a sequence of elongated ellipses of fixed volume, for which this ratio does explode. However, for classical perimeter, there are density estimates that, when combined with a perimeter restriction, imply that in the case of small mean curvature deficit the ratio $\diam(E)^n/|E|$ can be bounded above by a constant. It could be interesting to see if, under assumptions on the contact angle between $M$ and $\pa H$, there are density estimates along the lines of those in \cite{caffaroquesavin, millotsirewang} that would provide an upper bound in terms of $n$ and $P_s(E)$ for this term. }
\end{remark}

The statement of the theorem implies the following corollary:
\begin{corollary}\label{cor axially symmetric}
Let $s,\a,\s, E$ and $M$ be as in Theorem~\ref{thm axially symmetric}. If $\de_s(E)=0,$ i.e. if,
\[
\HH^s_E(q)=g(q_n)\,,
\]
for $q\in M\cap H$ and some function $g$ dependent only on the vertical direction, then $E$ is axially symmetric. 
\end{corollary}

\begin{remark}
{\rm The behavior of $g(t)$ as $t\to 0^+$ in Corollary \ref{cor axially symmetric} is determined by $\s$ in \eqref{sigma} according to \eqref{blow up}.}
\end{remark}

We now give additional context to Theorem \ref{thm axially symmetric}. Our work joins the efforts of many authors in understanding geometric variational problems in the fractional setting. This line of research was initiated in \cite{caffaroquesavin} with the study of Plateau problem with respect to the fractional perimeter:
\[
P_{s}(E)=I_s(E,E^c)\,,
\]
where
\[
I_s(E,F)=\int_E\int_F\frac{dxdy}{|x-y|^{n+s}}
\]
is the the fractional interaction energy of a pair of disjoint sets $E,F\subset\R^n$. (For further studies of nonlocal minimal surfaces see for instance \cite{AmbdPhiMart, barriosfigallivald,figvald,savinvald, daviladelpinodipievald, davdelpwei}.)
The fractional mean curvature operator $\HH^s_{E}$ defined in \eqref{fractional mean curvature} arises because, if $E$ is of class $C^2$ around a point $x\in\pa E$, then the first variation of $P_s(E)$ along the flow generated by a compactly supported and smooth vector-field $X$ satisfies
\[
\de P_s(E)[X]=\int_{\pa E}\,(X\cdot\nu_E)\,\HH^s_{E}\,d\H^{n-1}\,,
\]
(see \cite{figfuscmagmillmor, abatvald} for further results).
In this direction, the closest related result to Theorem~\ref{thm axially symmetric} is the recent extension of the classical rigidity theorem of Aleksandrov \cite{aleksandrov} (boundaries of compact sets with constant mean curvature are spheres) to fractional mean curvatures due to \cite{ciraolofigallimagginovaga,cabrefallweth1}. To be precise, in these papers it was shown that, if $E$ is a bounded open set with boundary of class $C^{2,\a}$ for some $\a\in(s,1)$ and such that $\HH^s_{E}$ is constant along $\pa E$, then $\pa E$ is a sphere. This was proven by adapting the original moving plane argument by Aleksandrov to the fractional setting. 

Following this classical moving plane method will also be our approach to Theorem~\ref{thm axially symmetric}. A moving hyperplane perpendicular to $\pa H$ is slid in a given direction until the reflected cap of $M$ across this hyperplane achieves a first contact point with $M$. In the well-known case of Aleksandrov's argument one has to discuss two kinds of tangency points, which become four different kinds of tangency points in Wente's work \cite{Wente80}, and also in our situation. Two of these four cases, where the tangency point is achieved away from $\pa H$, follow by repeating the arguments of \cite{ciraolofigallimagginovaga,cabrefallweth1}. However, the other two cases require new considerations because of the aforementioned degeneracy of the fractional mean curvature near the boundary hyperplane. Their proof is the main contribution of the paper.

As much as Wente's result is related to the study of classical capillarity theory, and in particular to the study of critical points for the Gauss free energy in the sessile droplet problem (see \cite{Finn} or \cite[Chapter 19]{maggiBOOK}), Theorem \ref{thm axially symmetric} can also be motivated by the study of a capillarity model using fractional perimeters to mimic surface tension. More precisely, in \cite{maggivaldinoci} the authors consider a fractional variant of the classical Gauss free energy for a droplet $E$ confined inside a container $\Om$,
\[
\E_s(E,\Om)=I_s(E,E^c\cap\Om)+\g\,I_s(E,\Om^c)\qquad E\subset\Om\,,
\]
where $s\in(0,1)$, $\g\in(-1,1)$. In \cite[Theorem 1.3, Theorem 1.4]{maggivaldinoci} it was proven that if $E\subset\Om$ is a volume-constrained minimizer of $\E_s(\cdot,\Om)$ such that $M=\ov{\pa E\cap\Om}$ is a $C^{1,\a}$-hypersurface with boundary for some $\s\in(s,1)$ and with $\bd(M)\subset\pa\Om$, then, for $c$ a constant, the Euler-Lagrange equation
\[
\HH_E^s(q)+(\g-1)\int_{\Om^c}\frac{1}{|x-q|^{n+s}}dx=c
\]
holds at each $q\in \Om\cap M$, together with a contact angle condition (fractional Young's law):
\[
\nu_E(x)\cdot\nu_\Om(x)=\cos(\pi-\theta(s,\g))\qquad\forall x\in M\cap\pa\Om=\bd(M)\,,
\]
where $\theta(s,0)=\pi/2$, $\theta(s,(-1)^+)=0$ and $\theta(s,1^-)=\pi$. In the corresponding sessile droplet problem, where one takes $\Om=H=\{x_n>0\}$, we end up in the situation considered in Corollary \ref{cor axially symmetric}, by setting
\begin{equation}
\label{g mv}
\HH_E^s(q)=-c+(1-\g)\int_{H^c}\frac{1}{|x-q|^{n+s}}dx\,.
\end{equation}
Because, in this case, $\HH_E^s$  is a function of the $q_n$ variable alone, Theorem \ref{thm axially symmetric} implies the axial symmetry of every volume-constrained critical point of the fractional Gauss free energy on a half-space.


\bigskip

\noindent {\bf Acknowledgments:} The author would like to thank Francesco Maggi for mentoring, guidance and helpful discussions. This work was supported by NSF-DMS Grants 1265910 and 1351122.

\section{Blowup of Fractional Mean Curvature}\label{blowup}
We summarize our basic notation and assumptions used in the paper. 
\medskip

\noindent\begin{minipage}[t]{0.25\textwidth}
\textbf{Assumption }(h1):
\end{minipage}
\begin{minipage}[t]{0.7\textwidth}
We assume that $s\in(0,1)$, $\a\in(s,1)$, $n\ge2$, and $\mbox{$E\subset\R^n$}$ is a bounded open connected subset of $H=\{x_n>0\}$ such that \mbox{$M:=\ov{\pa E\cap H}$} is a $C^{2,\a}$-hypersurface with boundary $\mbox{$\bd(M):=M\cap\pa H$}$.
\end{minipage}

\medskip

For $q\in M$ we let 
\begin{equation*}\label{defAq}
A_qM:=q+T_qM\,.
\end{equation*}
\begin{remark}
The normal vector $\nu_{E}$ on $M\cap H$ extends to $M\cap \pa H.$
\end{remark}
\begin{remark}{\rm
Under assumption (h1), there exist $\eta>0$ and $\g>0$ such that, for all $q\in M$, we have
\[
M\cap B_{\eta}(q)\subset P_{\eta,\g}\,,
\]
for 
\begin{equation}\label{parabola}
P_{\eta,\g}(q):=\{x\in B_\eta(q): |x-\p_q x|<\g |q-\p_q x|^{2+\a}\}\,,
\end{equation}
where $\p_q x$ is the projection of $x$ onto $A_qM$.}
\end{remark}

From \cite[Lemma 2.1]{ciraolofigallimagginovaga} we know that $\HH_E^{s}(q)$ is $C^1$ for all $q\in M\cap H$. But the following proposition shows that, as mentioned in the first section, if $q_n\rightarrow 0$, then $\HH_E^{s}(q)$ blows up like $q_n^{-s},$ notably without the assumption that $\nu_E(q)\cdot\nu_H(q)$ is constant.

\begin{proposition} \label{H blow} Let $E,s,\a$ be as in (h1). If $p\in M\cap \pa H,$ $\{q_m\}_{m\in\N}\subset M\cap H,$ and $q_m\rightarrow p$ as $m\rightarrow\infty$, 
\begin{itemize}
\item[(a)]then
\[
\HH_E^s(q_m)=\frac{c\big(n,s,\nu_E(p)\cdot \nu_H)\big)}{(q_m\cdot e_n)^s}\Big(1+{\rm O}(1)\Big)\,,
\]
as $m\rightarrow \infty.$
If in addition there is $\s\in(-1,1)$ such that $\nu_E(p)\cdot \nu_H=\sigma$ for all $p\in M\cap \pa H$, then 
\begin{equation}\label{with sig}
\HH_E^s(q)=\frac{c\big(n,s,\s\big)}{(q\cdot e_n)^s}\Big(1+{\rm O}(1)\Big)\,,
\end{equation}
as $q\cdot e_n\rightarrow 0,$ for $q\in M\cap H.$
\smallskip
\item[(b)] Moreover
\begin{equation}\label{prop b}
\limsup_{m\rightarrow \infty}\; (q_m\cdot e_n)^{s+1}|\nabla\HH_E^s(q_m)|\le c\big(n,s,\nu_E(p)\cdot \nu_H\big)\,.
\end{equation}
\end{itemize}
\end{proposition}
\begin{remark}
{\rm Note that in the proposition the proof of (a) only requires that $M\in C^{1,\a},$ but the proof of (b) requires that $M\in C^{2,\a}.$}\end{remark}
\begin{proof} Consider $q_m$ and $p$ as in the statement. By assumption (h1) there exists $\mbox{$\{p_m\}_{m\in\N}\subset M\cap \pa H$}$ such that $|p_m-q_m|\rightarrow 0$ as $m\rightarrow \infty$ and $A_{q_m}M\cap\pa H$ is parallel to $A_{p_m}M\cap \pa H$ for each $m$. Let $J_{q_m}$ be the wedge in $H$ bounded by $A_{q_m}M$ and $\pa H$ such that $\nu_{J_{q_m}}(q_m)=\nu_E(q_m).$ Similarly let $J_{p_m}$ be the wedge in $H$ bounded by $A_{p_m}M$ and $\pa H$ such that $\nu_{J_{p_m}}(p_m)=\nu_E(p_m).$ Let $J_{p_m}^*$ be the horizontal translation of $J_{p_m}$ such that $q_m\in J_{p_m}^*,$ see Figure~1. Let $\theta_m$ be defined so that 
\[
\cos(\theta_m)=\nu_{J_{p_m}^*}(q_m)\cdot \nu_{J_{q_m}}(q_m)\,.
\]
\begin{figure}
\includegraphics[scale=.55]{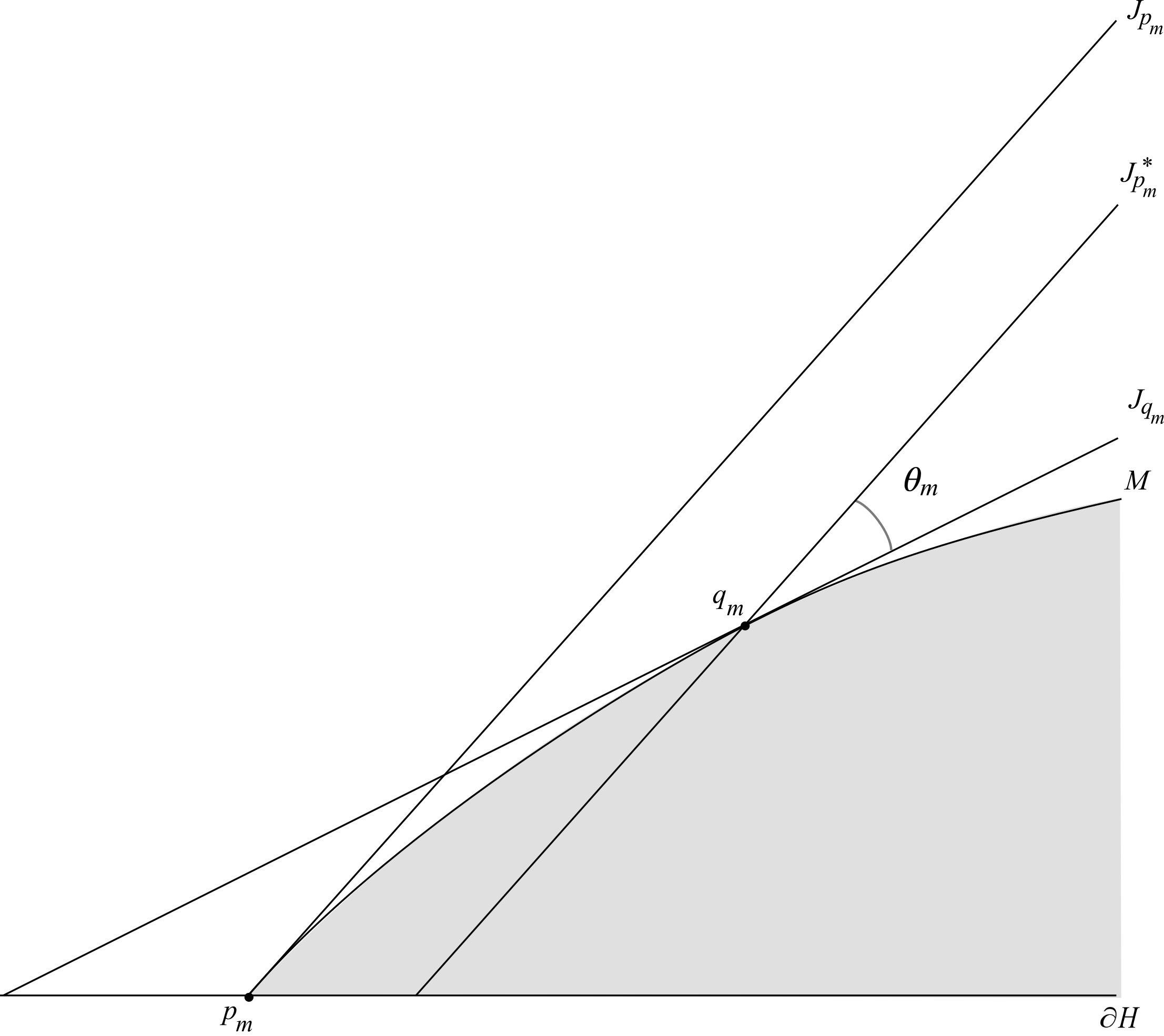}
 \caption{\small{The arrangement of the wedges in the case where $n=2$ and $\nu_E(p_m)\cdot \nu_H$ is negative. Here the shaded area represents the region inside of $E$.}}\label{wedges}
\end{figure} To prove (a), we will compare the mean curvature calculated at $q_m$ for $J_{p_m}^*$ and $J_{q_m}$ to understand $H_E^s(q_m)$ as $m\rightarrow\infty$. We first claim that 
\begin{equation}\label{first diff}
|\HH_{J_{q_m}}^s(q_m)-\HH_E^s(q_m)|\le \frac{K_0(n,s)}{\eta^s}+K_1(\eta,\a,\g)\,,
\end{equation}
for a constants $K_0(n,s)>0$ dependent only on $n$ and $s$, $K_1(\eta,\a,\g)>0$ dependent only on $\eta, \a$ and $\g$, and 
for $\eta$ as in \eqref{parabola}. In the following, the constant $K_0(n,s)$ may vary slightly from line to line, but for simplicity we will use the same letter. We use polar coordinates and find that
\begin{align}\label{outside wedge}
\int_{B_{\eta}(q_m)^c}\frac{\tilde{\chi}_{J_{q_m}}(x)}{|x-q_m|^{n+s}}dx
&\le n\om_n\int_\eta^{\infty}\frac{t^{n-1}}{t^{n+s}}dt\nonumber\\
&= \frac{K_0(n, s)}{\eta^s}\,,
\end{align}
and similarly
\begin{equation}\label{outside E}
\int_{B_{\eta}(q_m)^c}\frac{\tilde{\chi}_{E}(x)}{|x-q_m|^{n+s}}dx\le \frac{K_0(n, s)}{\eta^s}\,.
\end{equation}
By definition of $P_{\eta,\g}$ in \eqref{parabola}, we have
\begin{align}\label{in wedge vs E}
{\rm p.v.}\Big|\int_{B_{\eta}(q_m)}\frac{\tilde{\chi}_{J_{q_m}}(x)}{|x-q_m|^{n+s}}-\frac{\tilde{\chi}_{E}(x)}{|x-q_m|^{n+s}}dx\Big|&\le {\rm p.v.}\int_{P_{\eta,\g}(q_m)}\frac{dx}{|x-q_m|^{n+s}}\nonumber\\&=K_1(\eta,\a,\g)<\infty\,.
\end{align}
So by \eqref{outside wedge}, \eqref{outside E}, and \eqref{in wedge vs E} we have \eqref{first diff} as claimed.
We next show that
\begin{equation}\label{second diff}
|\HH_{J_{p_m}^*}^s(q_m)-\HH_{J_{q_m}}^s(q_m)|\le |\theta_m|\frac{K_0(n,s)}{(q_m\cdot e_n)^s}\,.
\end{equation}
Indeed, we know by symmetry that 
\begin{equation}\label{in wedges}
0={\rm p.v.}\int_{B_{q_m\cdot e_n}(q_m)}\frac{\tilde{\chi}_{J_{q_m}}(x)}{|x-q_m|^{n+s}}dz={\rm p.v.}\int_{B_{q_m\cdot e_n}(q_m)}\frac{\tilde{\chi}_{J_{p_m}^*}(z)}{|x-q_m|^{n+s}}dx\,,
\end{equation}
while 
\begin{equation}\begin{split}\label{out wedges}
\Big|\int_{B_{q_m\cdot e_n}(q_m)^c}\frac{\tilde{\chi}_{J_{q_m}}(x)}{|x-q_m|^{n+s}}-\frac{\tilde{\chi}_{J_{p_m}^*}(x)}{|x-q_m|^{n+s}}dx\Big|
&\le \int_{q_m\cdot e_n}^{\infty}\frac{|\theta_m|n\om_nr^{n-1}}{\pi r^{n+s}}dr\\
&=|\theta_m|\frac{K_0(n,s)}{(q_m\cdot e_n)^s}\,.
\end{split}\end{equation}
By \eqref{in wedges} and \eqref{out wedges} we have \eqref{second diff}. Together (\ref{first diff}) and (\ref{second diff}) imply that 
\begin{equation}\label{triangle wedge 1}
|\HH_{J_{p_m}^*}^s(q_m)-\HH^s_E(q_m)|\le \frac{K_0(n,s)}{\eta^s} +K_1(\eta,\a,\g)+|\theta_m|\frac{ K_0(n,s)}{(q_m\cdot e_n)^s}\,.
\end{equation}
By scaling (see for example \cite[Theorem 1.4]{maggivaldinoci}) we have
\begin{equation}\label{wedge formula}
\HH_{J_{p_m}^*}^s(q_m)=\frac{c(n,s,\nu_{J_{p_m}^*}(q_m)\cdot\nu_H)}{(q_m\cdot e_n)^s}\,,
\end{equation}
where $c(n,s,\cdot)$ is continuous.
Since $|\theta_m|\rightarrow 0$ as $m\rightarrow\infty$ we know $|\nu_{J_{p_m}^*}(q_m)-\nu_{E}(p_m)|\rightarrow 0$ as $m\rightarrow\infty$.  Therefore, because $\nu_{E}(p_m)\rightarrow\nu_E(p)$ as $m\rightarrow\infty,$ \eqref{triangle wedge 1} and \eqref{wedge formula} imply that
\[
\HH^s_E(q_m)=\frac{c\big(n,s,\nu_E(p)\cdot\nu_H\big)}{(q_m\cdot e_n)^s}\Big(1+{\rm O}(1)\Big)\,,
\]~as $m\rightarrow\infty.$ Moreover, in the case that $\nu_E(p)\cdot\nu_H=\s$ for all $p\in M\cap\pa H$, this implies \eqref{with sig}.

The proof of (b) closely parallels that of part (a), with extra attention to convergence of $\na \HH_E^s$ throughout. Fix $m\in \N$ large. We approximate the kernel $|t|^{-(n+s)}$ by $\var_{\e}(t)\in C_c^{\infty}([0,\infty))$ such that $\var_\e\ge 0, \var'_\e\le 0,$ and
\[\left\{
\begin{array}{l}
t^{n+s}\var_\e(t)+t^{n+s+1}|\var'_\e(t)|\le C(n,s)\,,\\
|\var'_\e(t)|\uparrow \frac{n+s}{t^{n+s+1}} \qquad \mbox{ as }\e\rightarrow 0^+\,,
\end{array}
\right.\quad\forall t> 0\,.
\]
By construction $\var_\e(t)\uparrow |t|^{-(n+s)}$ for all $t>0.$ For $q\in M\cap \pa H$, we set
\[
\HH_E^{s,\e}(q):=\int_{\R^n}\tilde{\chi}_{E}(x)\var_\e(|x-q|)dx\,,
\] and let $u_{\e,q}(x)=\var_\e(|x-q|)$.  We have
\begin{equation*}
\nabla \HH_E^{s,\e}(q)=\int_{\R^n}\tilde{\chi}_E(x)\nabla u_{\e,q}(x) dx\,.
\end{equation*} Note that, by \cite{ciraolofigallimagginovaga}, 
\begin{equation}\label{convergence} 
\lim_{\e\rightarrow0}\nabla H_E^{s,\e}(q)=\nabla H_E^{s}(q)\qquad \forall q\in M\cap H\,.
\end{equation}
To prove \eqref{prop b} it suffices to show that 
\begin{equation}\label{real b}
\limsup_{\e\rightarrow 0}(q_m\cdot e_n)^{s+1}|\na H_E^{s,\e}(q_m)|\le  c\big(n,s,\nu_E(p)\cdot \nu_H\big)\,,
\end{equation} 
for large enough $m.$
First we will show that
\begin{equation}\label{derivative triangle I}
\limsup_{\e\rightarrow 0}|\nabla \HH_{J_{q_m}}^{s,\e}(q_m)-\nabla \HH_E^{s,\e}(q_m)|\le \frac{K(n)}{\eta^{s+1}}+\tilde{K}(\eta,\a,\g)\,,
\end{equation}
for some $K(n)$ dependent only on $n$ and $\tilde{K}(\eta,\a,\g)>0$ dependent only on $\eta, \a$ and $\g.$  In the following, the constant $K(n)$ may vary slightly from line to line, but for simplicity we will use the same letter. We split 
\begin{equation}\begin{split}
\label{derivative triangle I 2}
\Big|\int_{\R^n}(\tilde{\chi}_{E}-\tilde{\chi}_{J_{q_m}})(x)\nabla u_{\e,q_m}(x)dx\Big|
\le&\Big|\int_{B_{\eta}(q_m)^c}(\tilde{\chi}_{E}-\tilde{\chi}_{J_{q_m}})(x)\nabla u_{\e,q_m}(x) dx\Big|\\
&+\Big|\int_{B_{\eta}(q_m)}(\tilde{\chi}_{E}-\tilde{\chi}_{J_{q_m}})(x)\nabla u_{\e,q_m}(x) dx\Big|\,.\qquad
\end{split}\end{equation}
By applying the monotone convergence theorem in $\e$ and by using polar coordinates we see that there is a constant $K(n)$ dependent only on $n$ such that 
\begin{align}\label{outside wedge dir}
\limsup_{\e\rightarrow0}\Big|\int_{B_{\eta}(q_m)^c}\tilde{\chi}_{J_{q_m}}(x)\nabla u_{\e,q_m}(x) \;dx\Big|&\le
\limsup_{\e\rightarrow0}\int_{B_{\eta}(q_m)^c}|\nabla u_{\e,q_m}(x)| \;dx\nonumber\\
&\le K(n)\int_\eta^{\infty}\frac{t^{n-1}}{t^{n+s+1}} \; dt\nonumber\\
&=\frac{ K(n)}{\eta^{s+1}}\,,
\end{align}
and similarly
\begin{equation}\label{outside E dir}
\limsup_{\e\rightarrow 0}\Big|\int_{B_{\eta}(q_m)^c}\tilde{\chi}_{E}(x)\nabla u_{\e,q_m}(x)dx\Big|\le\frac{K(n)}{\eta^{s+1}}\,.
\end{equation}
By definition of $P_{\eta,\g}$ in \eqref{parabola}, we have
\begin{align}\label{pinchie}
 \Big|\int_{B_{\eta}(q_m)}(\tilde{\chi}_{E}-\tilde{\chi}_{J_{q_m}})(x)\nabla u_{\e,q_m}(x) dx\Big|
&\le\Big|\int_{P_{\eta,\g}(q_m)}(\tilde{\chi}_{E}-\tilde{\chi}_{J_{q_m}})(x)\nabla u_{\e,q_m}(x) dx\Big|\nonumber\\
&\le\int_{P_{\eta,\g}(q_m)}2|\vphi'_\e(|x-q_m|)| dx\,.
\end{align}
Therefore, because 
\[
 \lim\limits_{\e\rightarrow 0}\int_{P_{\eta,\g}(q_m)}2|\vphi'_\e(|x-q_m|)|\; dx= {\rm p.v.}\int_{P_{\eta,\g}(q_m)}\frac{2dx}{|x-q_m|^{n+s+1}}=\tilde{K}(\eta,\a,\g)<\infty\,,
\]
and by the monotone convergence theorem in $\e$, we can plug \eqref{outside wedge dir}, \eqref{outside E dir} and \eqref{pinchie} into \eqref{derivative triangle I 2} and get \eqref{derivative triangle I}. 
Next we prove that 
\begin{equation}\label{second diff II}
\limsup_{\e\rightarrow0}|\na\HH_{J_{p_m}^*}^{s,\e}(q_m)-\na\HH_{J_{q_m}}^{s,\e}(q_m)|\le |\theta_m|\frac{K(n)}{(q_m\cdot e_n)^{s+1}}\,.
\end{equation}
We know, by symmetry, that, for any $\e$,
\begin{equation}\label{in wedges diff}
0=\int_{B_{q_m\cdot e_n}(q_m)}\tilde{\chi}_{J_{q_m}}(x)\nabla u_{\e,q_m}(x) \;dx=\int_{B_{q_m\cdot e_n}(q_m)}\tilde{\chi}_{J_{p_m}^*}(x)\nabla u_{\e,q_m}(x) \;dx\,,
\end{equation}
so, 
\begin{align}
\Big|\int_{\R^n}(\tilde{\chi}_{J_{q_m}}-\tilde{\chi}_{J_{p_m}^*})(x)\nabla u_\e(x)dx\Big|
&=\Big|\int_{B_{q_m\cdot e_n}(q_m)^c}(\tilde{\chi}_{J_{q_m}}-\tilde{\chi}_{J_{p_m}^*})(x)\nabla u_\e(x)dx\Big|\nonumber\\
&\le\int_{B_{q_m\cdot e_n}(q_m)^c}|\tilde{\chi}_{J_{q_m}}-\tilde{\chi}_{J_{p_m}^*}|(x)\big|\vphi'_\e(|x-q_m|)\big|dx\,.\nonumber
\end{align}
Therefore, by monotone convergence theorem in $\e$, there is a constant $K(n)$ such that 
\begin{equation}\label{out wedges dif}
\limsup_{\e\rightarrow 0}\Big|\int_{\R^n}(\tilde{\chi}_{J_{q_m}}-\tilde{\chi}_{J_{p_m}^*})(x)\nabla u_\e(x)dx\Big|     \le|\theta_m|\frac{K(n)}{(q_m\cdot e_n)^{s+1}}\,,
\end{equation}
which implies \eqref{second diff II}. Together, \eqref{derivative triangle I} and \eqref{second diff II} imply that 
\begin{equation}\label{triangle wedge}
\limsup_{\e\rightarrow0}|\na \HH_{J_{p_m}^*}^{s,\e}(q_m)\cdot \t_m-\na \HH^{s,\e}_E(q_m)\cdot \t_m|\le \frac{K(n)}{\eta^{s+1}}+\tilde{K}(\eta,\a,\g)+|\theta_m|\frac{K(n)}{(q_m\cdot e_n)^{s+1}}\,.
\end{equation} 
For $m$ large we use \eqref{convergence}, exploit symmetry as in \eqref{in wedges diff}, apply monotone convergence theorem in $\e$, and use polar coordinates as in \eqref{outside wedge dir}, to find that there is a constant $c(n,s,\nu_E(p_m)\cdot\nu_H )>0$ dependent on $n,$ $s$, and $\nu_E(p_m)\cdot\nu_H$, such that 
\begin{align}\label{dif wedge formula}
|\na\HH_{J_{p_m}^*}^s(q_m)|&=\lim\limits_{\e\rightarrow 0}\Big|\int_{\R^n}\tilde{\chi}_{J_{p_m}^*}(x)\nabla u_{\e,q_m}(x) dx\Big|\nonumber\\
&=\lim\limits_{\e\rightarrow 0}\Big|\int_{B_{q_m\cdot e_n}(q_m)^c}\tilde{\chi}_{J_{p_m}^*}(x)\nabla u_{\e,q_m}(x) dx\Big|\nonumber\\
&\le \int_{B_{q_m\cdot e_n}(q_m)^c}\frac{(n+s)|\tilde{\chi}_{J_{p_m}^*}(x)|}{|x-q_m|^{n+s+1}}dx\nonumber\\
&\le\frac{ c\big(n,s,\nu_E(p_m)\cdot\nu_H\big)}{(q_m\cdot e_n)^{s+1}}\,,
\end{align}
as $m\rightarrow \infty,$ where $c(n,s,\cdot)$ is continuous.
Letting $m\rightarrow\infty$ in \eqref{triangle wedge} and \eqref{dif wedge formula}, by $|\theta_m|\rightarrow 0$ and $\nu_{E}(p_m)\rightarrow\nu_E(p)$ as $m\rightarrow\infty$, we have 
\[
\limsup_{m\rightarrow \infty}(q_m\cdot e_n)^{s+1}\HH^s_E(q_m)\le c\big(n,s,\nu_E(p)\cdot\nu_H\big)\,.
\] 
\end{proof}

\section{Proof of Almost Axial Symmetry}

In this section we prove part (a) of Theorem \ref{thm axially symmetric}. Recall that in addition to assumption (h1) we assume the existence of $\s\in(-1,1)$ such that 
\begin{equation*}
\nu_E(q)\cdot\nu_H(q)=\s\qquad\forall q\in M\cap\pa H\,.
\end{equation*}

\begin{proof}[Proof of Theorem 1.1 (a)]  We start by introducing notation used in the proof. Let  $e\in S^{n-1}\cap\pa H$. Without loss of generality we assume that $e=e_1.$ We define
\[
\begin{array}{rlll}
\pi_\mu&=&\{x_1=\mu\}& \, \mbox{a hyperplane perpendicular to $e_1$}\,,\\
\rho_\mu(x)&=&(2\l-x_1,...,x_n)&\, \mbox{the reflection of $x$ across $\pi_\mu$}\,,
\end{array}
\]
and we set 
\begin{equation}\label{lambda}
\l=\inf\big\{\mu\in\R:\rho_\mu(E)\cap \{x_1< \mu\}\subset E \big\} 
\end{equation}
to be the critical value for $\mu$ for the moving planes argument. By regularity of the boundary, we know $\l$ is well defined. We call $\pi_{\l}$ the critical hyperplane and, as long as there is no confusion, we will denote
\[
\rho(x)=\rho_{\l}(x)\,.
\]
As in \cite{Wente80}, at least one of the four following cases holds in the critical position (see Figure~2):
\begin{MYitemize}
\item\textit{ Case one:} $\pa \rho(E)$ is tangent to $\pa E$ at some point $p\in M\cap \pi_{\l}^c\cap\bd(M)^c\,$.

\item\textit{ Case two:} $\pa \rho(E)$ is tangent to $\pa E$ at some point $p\in M\cap\pi_{\l}\cap\bd(M)^c\,$.

\item\textit{ Case three:} $\pa \rho(E)$ is tangent to $\pa E$ at some point $p\in M\cap\pi_{\l}^c\cap\bd(M)\,$.

\item\textit{ Case four:} $\pa \rho(E)$ is tangent to $\pa E$ at some point $p\in M\cap\pi_{\l}\cap\bd(M)\,$.
\end{MYitemize}

\begin{figure}
\includegraphics[scale=.45]{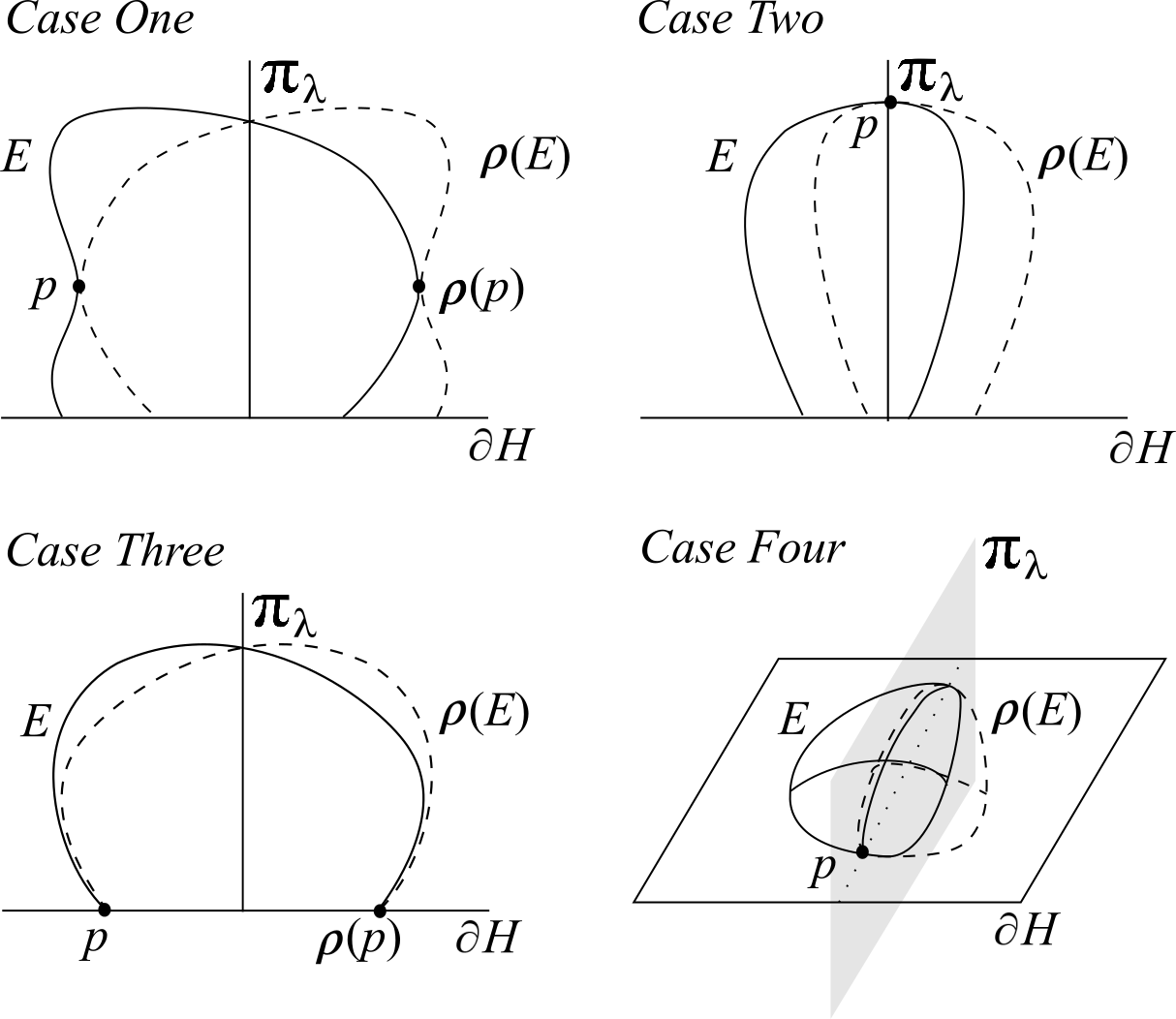}
 \caption{\small{The four cases of the moving planes argument.}}\label{cases}
\end{figure}
\smallskip

To prove almost symmetry our goal is to show that $|E\Delta \rho(E)|$ is bounded above by a multiple of $\sqrt{\de_s(E)}$ in each case. In particular we will prove that
\begin{equation}\label{int diff}
\int_{E\Delta \rho(E)}\dist(x,\pi_\l)dx\le \frac{5}{2(n+s)}\diam(E)^{n+1}\de_s(E)\,.
\end{equation} 
Case one and case two follow \cite{ciraolofigallimagginovaga}, but we include the proofs for the sake of completeness.

\medskip
\noindent\textit{Case one}:
Suppose there is $p\in\{x_1<\l\}$ such that $p\in M\cap \rho(M)\cap H.$ Then
\begin{equation*}
\begin{split}
\HH_E^s(\rho(p))-\HH_E^s(p)&=\HH_{\rho(E)}^{s}(p)-\HH_E^{s}(p)\\
&=2\Big( \int_{E\setminus \rho(E)}\frac{1}{|x-p|^{n+s}}dx-\int_{\rho(E)\setminus E}\frac{1}{|x-p|^{n+s}}dx \Big)\\
&=2\int_{E\setminus \rho(E)}\Big(\frac{1}{|x-p|^{n+s}}-\frac{1}{|\rho(x)-p|^{n+s}}\Big)dx \,.
\end{split}
\end{equation*}
Denote $x=(x_1,\bar{x}).$ By definition $\rho(x)=(2\l-x_1,\bar{x}),$ so 
\begin{equation}\begin{split}\label{case one calc two}
\Big(\frac{|\rho(x)-p|}{|x-p|}\Big)^{2}&=\frac{(2\l-2p_1-(x_1-p_1))^2+(\bar{x}-\bar{p})^2}{|x-p|^2}\\
&=1+ \frac{(2\l-2p_1)^2-2(2\l-2p_1)(x_1-p_1)}{|x-p|^2}\\
&=1+ \frac{4(\l-p_1)(\l-x_1)}{|x-p|^2}\,,
\end{split}\end{equation}
and thus
\begin{equation}\label{case 1 mid}\begin{split}
\frac{1}{|x-p|^{n+s}}-\frac{1}{|\rho(x)-p|^{n+s}}&=\frac{1}{|\rho(x)-p|^{n+s}}\Big[\Big(\frac{|\rho(x)-p|}{|x-p|}\Big)^{n+s}-1\Big]\\
&=\frac{1}{|\rho(x)-p|^{n+s}}\Big[\Big(1+\frac{4(\l-x_1)(\l-p_1)}{|x-p|^2}\Big)^{\frac{n+s}{2}}-1\Big]\,.
\end{split}
\end{equation}
Therefore, by the convexity of $f(t)=(1+t)^{(n+s)/2}-1$, we know that
\begin{equation}\label{case one pos}
\frac{1}{|x-p|^{n+s}}-\frac{1}{|\rho(x)-p|^{n+s}}\ge\frac{2(n+s)(\l-x_1)(\l-p_1)}{|\rho(x)-p|^{n+s}|x-p|^2}\ge\frac{2(n+s)(\l-x_1)(\l-p_1)}{\diam(E)^{n+s+2}}\ge0\,,
\end{equation}
for $x\in E\setminus \rho(E)$. Indeed, for $x\in E\setminus \rho(E),$ we have $|x-p|=|\rho(x)-p|\le\diam(E)$, $\l-x_1\ge0$, and $|p-\rho(p)|=2(\l-p_1)>0$. 
So 
\begin{equation*}\begin{split}
\frac{\de_s(E)}{\diam(E)^{s+1}}&\ge\frac{|\HH_E^s(\rho(p))-\HH_E^s(p)|}{|\rho(p)-p|}= \frac{|\HH_E^s(\rho(p))-\HH_E^s(p)|}{2(\l-p_1)}\\
&\ge\frac{2(n+s)}{\diam(E)^{n+s+2}}\int_{E\setminus \rho(E)}|x_1-\l |\,,
\end{split}\end{equation*}
which implies 
\[
\int_{E\Delta \rho(E)}|x_1-\l |\le\frac{1}{(n+s)} \diam(E)^{n+1}\de_s(E)\,,
\]
and therefore \eqref{int diff} in case one.

\medskip

\noindent\textit{Case two:} Suppose that there is $p\in\{x_1=\l\}\cap H$ such that $M$ is orthogonal to $\pi_\l$. As in the proof of Proposition \ref{H blow}, we approximate the kernel $|t|^{-(n+s)}$ by $\var_{\e}(t)\in C_c^{\infty}([0,\infty))$ such that $\var_\e\ge 0, \var'_\e\le 0,$ and
\[\left\{
\begin{array}{l}
t^{n+s}\var_\e(t)+t^{n+s+1}|\var'_\e(t)|\le C(n,s)\,,\\
|\var'_\e(t)|\uparrow \frac{n+s}{t^{n+s+1}} \qquad \mbox{ as }\e\rightarrow 0^+\,,
\end{array}
\right.\quad\forall t> 0\,.
\]
Then, by construction, $\var_\e(t)\uparrow |t|^{-(n+s)}$ for all $t>0.$ Recall that for any set $F\in \R^n$ we set $\tilde{\chi}_F(x)=\chi_{F^c}(x)-\chi_F(x).$ We define 
\[
\HH_E^{s,\e}(q):=\int_{\R^n}\tilde{\chi}_{E}(x)\var_\e(|x-q|)dx\,,
\] and let $u_\e(x)=\var_\e(|x-p|)$.
For all $a>0$, we know, by \cite{ciraolofigallimagginovaga}, that $\HH_E^{s,\e}\rightarrow\HH_E^{s}$ in $C^{1}(\pa E\cap \{x_n>a\})$ as $\e\rightarrow 0$. So we have
\begin{equation*}
\lim\limits_{\e\rightarrow 0}\nabla \HH_E^{s,\e}(p)\cdot e_1=\lim\limits_{\e\rightarrow 0}\int_{\R^n}\tilde{\chi}_E(x)\nabla u_\e(x)\cdot e_1dx\,.
\end{equation*}
Because $\nabla u_\e(x)\cdot e_1=\var'_\e(|x-p|)\frac{(x-p)\cdot e_1}{|x-p|}$ is odd with respect to the hyperplane $\{x_1=\l \}$, we know that
\[
\int_{ H^c}\nabla u_\e(x)\cdot e_1dx=0, \quad \int_{E\cap \rho(E)}\nabla u_\e(x)\cdot e_1dx=0, \quad \text{and } \quad \int_{(E\cup \rho(E))^c\cap H}\nabla u_\e(x)\cdot e_1dx=0\,.
\]
Therefore
\begin{equation*}
\begin{split}
\nabla \HH_E^{s,\e}(p)\cdot e_1&=\int_{\R^n}[\chi_{E^c}-\chi_{E}](x)\nabla u_\e(x)\cdot e_1dx\\
&=\int_{\R^n}\Big[\chi_{H^c}+\chi_{(E\cup \rho(E))^c\cap H}+\chi_{\rho(E)\setminus E}-\chi_{E\setminus \rho(E)}-\chi_{E\cap \rho(E)}\Big](x)\nabla u_\e(x)\cdot e_1dx\\
&=-2\int_{E\setminus \rho(E)}\nabla u_\e(x)\cdot e_1dx\,.
\end{split}
\end{equation*}
So, by the negativity of $\frac{(x-p)\cdot e_1}{|x-p|}$ for $x\in E\setminus \rho(E)\subset \{x_1<\l \}$ and by the monotone convergence of $|\varphi_\e'(t)|$ in $\e$, we have
\begin{equation*}
\nabla \HH_E^{s}(p)\cdot e_1=-2(n+s)\int_{E\setminus \rho(E)}\frac{(p-x)\cdot e_1}{|x-p|^{n+s+2}}dx\,.
\end{equation*}
Therefore, because $|\nabla \HH_E^s(p)\cdot e_1|\le \diam(E)^{-(s+1)}\de_s(E)$, and by
\[
\frac{(p-x)\cdot e_1}{|x-p|^{n+s+2}}\ge \frac{|x_1-\l |}{\diam(E)^{n+s+2}}\qquad \forall x\in E\setminus\rho(E)\,,
\]
we have that
\[
\frac{\de_s(E)}{\diam(E)^{s+1}}\ge\frac{2(n+s)}{\diam(E)^{n+s+2}}\int_{E\setminus \rho(E)}|x_1-\l |dx\,.
\]

\medskip

\noindent\textit{Case three}: We now consider the case where $p\in\bd(M)\cap\{x_1<\l\}$. As in case one we would like to consider the difference
\[
\HH^s_{\rho(E)}(p)-\HH^s_E(p)\,,
\]
but both terms equal infinity for $p\in\bd(M)$. So we use an approximation argument to improve our understanding of this difference as a limit. A key tool for this approximation, and also for a similar approximation in case four, will be the local graphicality of $M$ and $\rho(M)$ over $A_pM=A_p\rho(M)$ around $p.$ (Recall from Section \ref{blowup} that $A_qM:=q+T_qM$ for $q\in M$.) Let $U\subset A_p\cap H$ be a local neighborhood of $p$ in $A_pM\cap H$ such that we can define $v: U\rightarrow M$ to be the normal map along $\nu_E(p)$ to $M$. There is a set $U^*\subset U$ such that $v^*: U^* \rightarrow\rho( M)$ can be defined as the normal map along $\nu_E(p)$ to $\rho(M)$. If $U^*\ne U$ then we can reset it so that the sets are equal, because $U^*$ contains a neighborhood of $p$ in $A_pM$  on which $v$ may be defined. We can expand $v$ and $v^*$ as
\[
v(\hat q)=v(p)+\nabla v(p)[\hat{q}-p] +O(|\hat{q}-p|^{2})=p+\nabla v(p)[\hat{q}-p] +O(|\hat{q}-p|^{2})
\]
and
\[
v^*(\hat q)=v^*(p)+\nabla v^*(p)[\hat{q}-p]+O(|\hat{q}-p|^{2})=p+\nabla v^*(p)[\hat{q}-p]+O(|\hat{q}-p|^{2})\,,
\]
for $\hat{q}\in A_pM\cap H$ with $|\hat{q}-p|$ small enough.  Let 
\[
\hat{q}_m:=p-\frac{\nu_{co}^M(p)}{m} \qquad\mbox{for $m\in\N$}\,,
\] 
where $\nu_{co}^M(p)$ is the conormal vector for $p$ with respect to $M.$ Note that this conormal vector is also the same for $\rho(M)$ because $p$ is a point of tangency between $M$ and $\rho(M)$, $M$ is $C^{2,\a}$, and $M$ has constant contact angle with $\pa H.$
Let $q_m=v(\hat{q}_m)$ and $q_m^*=v^*(\hat{q}_m)$. We can now consider the correct analogue of $\HH^s_{\rho(E)}(p)-\HH^s_E(p)$ to be the limit of
\begin{equation}\label{case 3 diff}
Q_m:=\HH_{\rho(E)}^{s}(q_m^*)-\HH_E^{s}(q_m)\,,
\end{equation}
as $m\rightarrow\infty.$

\smallskip

  \noindent {\it Step one}: We show that 
  \begin{equation}\label{*}
  \limsup_{m\rightarrow\infty}|Q_m|\le  \frac{5}{2}|\l-p_1|\frac{\de_s(E)}{\diam(E)^{s+1}}
  \end{equation}
 as $m\rightarrow \infty$. The claim \eqref{*} is trivial if $\s= 0$ because then $q_m$ and $q_m^*$ have the same vertical component. Indeed, by definition of $\de_s$, we then have
 \[
| Q_m|\le \frac{\de_s(E)}{\diam (E)^{s+1}}|q_m-\rho(q_m)|\,,
 \]
which implies \eqref{*} because $|q_m-\rho(q_m)|\rightarrow |p-\rho(p)|=2|p_1-\l|$ as $m\rightarrow \infty$.
 Otherwise, let 
 \[
 e_p:=\frac{\nu_E(p)-(\nu_E(p)\cdot e_n)e_n}{|\nu_E(p)-(\nu_E(p)\cdot e_n)e_n|}\,.
 \]
 For $m$ large enough, and for each $q_m^*\in \rho(M)$ close enough to $p$, let $\tilde{q}_m\in M$ be defined to be the unique projection of $q_m^*$ onto $M$ along $e_p.$ Note that the definition of $\tilde{q}_m$ enforces that $q_m,q_m^*,\tilde{q}_m,$ and $\hat{q}_m$ are all contained in the plane $T:=\Span\{\nu_E(p),\nu_{co}^M(p)\}$, and by the continuity of the projection we have that $\tilde{q}_m\rightarrow p$ as $m\rightarrow\infty.$ 
 
 By the triangle inequality
\begin{equation}\label{**}
|Q_m|=\big|\HH_{\rho(E)}^{s}(q_m^*)-\HH_E^{s}(q_m)\big|\le\big|\HH_{\rho(E)}^{s}(q_m^*)-\HH_{E}^{s}(\tilde{q}_m)\big|+\big|\HH_{E}^{s}(\tilde{q}_m)-\HH_E^{s}(q_m)\big|\,. 
\end{equation}
By definition of $\de_s(E)$, we know
\[
\big|\HH_{\rho(E)}^{s}(q_m^*)-\HH_{E}^{s}(\tilde{q}_m)\big|=\big|\HH_E^{s}(\rho(q_m^*))-\HH_{E}^{s}(\tilde{q}_m)\big|\le |\rho(q_m^*)-\tilde{q}_m|\frac{\de_s(E)}{\diam(E)^{s+1}}\,,
\]
for any given $m$. But as $m\rightarrow \infty$ we know $ |\rho(q_m^*)-\tilde{q}_m|\rightarrow |\rho(p)-p|=2|\l-p_1|$. So for $m$ large enough 
\[
\big|\HH_{\rho(E)}^{s}(q_m^*)-\HH_{E}^{s}(\tilde{q}_m)\big|\le \frac{5}{2}|\l-p_1|\frac{\de_s(E)}{\diam(E)^{s+1}}.
\]
Therefore to prove \eqref{*} it suffices to show that 
\begin{equation}\label{*2}
\limsup_{m\rightarrow 0}|\HH_{E}^{s}(\tilde{q}_m)-\HH_E^{s}(q_m)|=0\,.
\end{equation}
By Proposition \ref{H blow}(b), there is a constant $c= c\big(n,s,\s\big)>0$ such that, if $h_m=\min\{q_m\cdot e_n, \tilde{q}_m\cdot e_n\}$, then
\begin{equation}\label{unif dif apl}
|\HH_{E}^{s}(\tilde{q}_m)-\HH_E^{s}(q_m)|\le c\big(n,s,\s\big)\frac{  |q_m-\tilde{q}_m|}{h_m^{s+1}}\,.
\end{equation}
We claim that
\begin{equation}\label{case three step one}
\lim_{m\rightarrow\infty}\frac{|q_m-\tilde{q}_m|}{|h_m|^{1+s}}=0\,,
\end{equation} 
which would suffice to prove \eqref{*2}.
 Because $\tilde{q}_m\in M$, for $m$ large enough there is $\hat{t}_m\in A_pM$ such that $v(\hat{t}_m)=\tilde{q}_m.$ Fix $m\in\N$ large. Let  $B_m=\max\{|\hat{q}_m-p|,|\hat{t}_m-p|\}$ and $\mbox{$b_m=\min\{|\hat{q}_m-p|,|\hat{t}_m-p|\}$}.$ So $B_m-b_m=|\hat{t}_m-\hat{q}_m|$. 
 To prove \eqref{case three step one} we will show that there exist constants $C(\s), C_1(\s)>0$, dependent only on $\s,$ such that 
   \begin{equation}\label{numerator}
 |q_m-\tilde{q}_m|\le C(\s)B_m^{2+\a}\,,
 \end{equation}
 and
 \begin{equation}\label{denominator}
 h_m\ge C_1(\s)\Big(B_m-C(\s)B_m^{2+\a}\Big)\,.
 \end{equation}
Together \eqref{numerator} and \eqref{denominator} give
 \begin{equation*}\begin{split}
 \limsup_{m\rightarrow\infty}\frac{|q_m-\tilde{q}_m|}{|h_m|^{1+s}}&\le \lim_{m\rightarrow\infty} \frac{C(\s)B_m^{2+\a}}{C_1(\s)(B_m-C(\s)B_m^{2+\a})^{s+1}}\\&= \lim_{m\rightarrow\infty}\frac{C(\s)B_m^{2+\a}}{C_1(\s)B_m^{1+s}(1-C(\s)B_m^{1+\a})^{1+s}}=0\,,
 \end{split}\end{equation*}
 which implies \eqref{case three step one} because $B_m\rightarrow 0$ as $m\rightarrow \infty$.

To prove  \eqref{numerator} and \eqref{denominator} we will exploit local graphicality of $M$ and $\rho(M)$ over $A_pM$, considered in $T$. In this plane, in a neighborhood of $p,$ there is an ordering of $M$, $\rho(M)$, and $A_pM$ with respect to $e_p.$ In particular, if we let $\hat{a}_m$ be the projection of $q_m^*$ along $e_p$ to $A_pM$, then $\tilde{q}_m\cdot e_p $, $q_m^*\cdot e_p,$ and $\hat{a}_m\cdot e_p$ have a fixed order for large enough $m$. Because the moving planes method forces $\tilde{q}_m\cdot e_p \ge q_m^*\cdot e_p,$ there are three possible orderings of these coordinates. Combined with the fact that we can consider $\s>0$ or $\s<0$, we have six total subcases to consider, see Figure 3. Because each case takes place in $T$ we set up the following notation. Define
\[
f(d):=v(p-d\nu_{co}^M(p))\cdot\big(\nu_E(p)\big)\,,
\qquad\mbox{ and }\qquad
f^*(d):=v^*(p-d\nu_{co}^M(p))\cdot\big(\nu_E(p)\big)\,.
\]
Because $v$ and $v^*$ are $C^{2,\a}$, with $v(p)=v^*(p)=0$, and with $\nabla v(p)\cdot\nu_E(p)=\nabla v^*(p)\cdot\nu_E(p)=0$, there exists a constant $\g>0$ such that, $f$ and $f^*$ are monotone, and 
\begin{equation}\label{f reg}
|f(d)|,|f^*(d)|\le \g d^{2+\a},
\end{equation}
for $d$ small enough. Moreover for $m$ large enough, there is a constant $c\le1$ 
such that 
\begin{equation}\label{lip bound f}
|f(d_2)-f(d_1)|\le c|d_2-d_1|\qquad \forall d_1,d_2\mbox{ small enough}.
\end{equation}
Let $\tilde{f}(d):=f(d)+f^*(d).$ 
We first prove \eqref{numerator} and then \eqref{denominator}. By \eqref{lip bound f}, we have
\begin{equation}\label{real num}
 |q_m-\tilde{q}_m|=\sqrt{(B_m-b_m)^2+(f(B_m)-f(b_m))^2}\le 2|B_m-b_m|\,,
\end{equation}
so we only need to prove there is a constant $C(\s)$ dependent only on $\s$ such that 
\begin{equation}\label{real nume bound}
|B_m-b_m|\le C(\s)\tilde{f}(B_m) \le C(\s)B_m^{2+\a}\,,
\end{equation}
and we have \eqref{numerator}. In each subcase the proof comes from considering one or two right triangles in $T$ with one edge parallel to $A_pM$, one edge perpendicular to $A_pM$, and one edge parallel to $\pa H,$ where the last edge is contained in or equal to the line segment connecting $q_m^*$ and $\tilde{q}_m$. Define $\theta$ so that 
\[
\cos(\theta)=\nu_E(p)\cdot\nu_H=\s\,.
\]
  In the first three cases, where $\s>0$, note that $\theta<\pi/2$, and for the last three, where $\s<0$, note that $\theta>\pi/2$. 

\begin{figure}
\includegraphics[scale=.68]{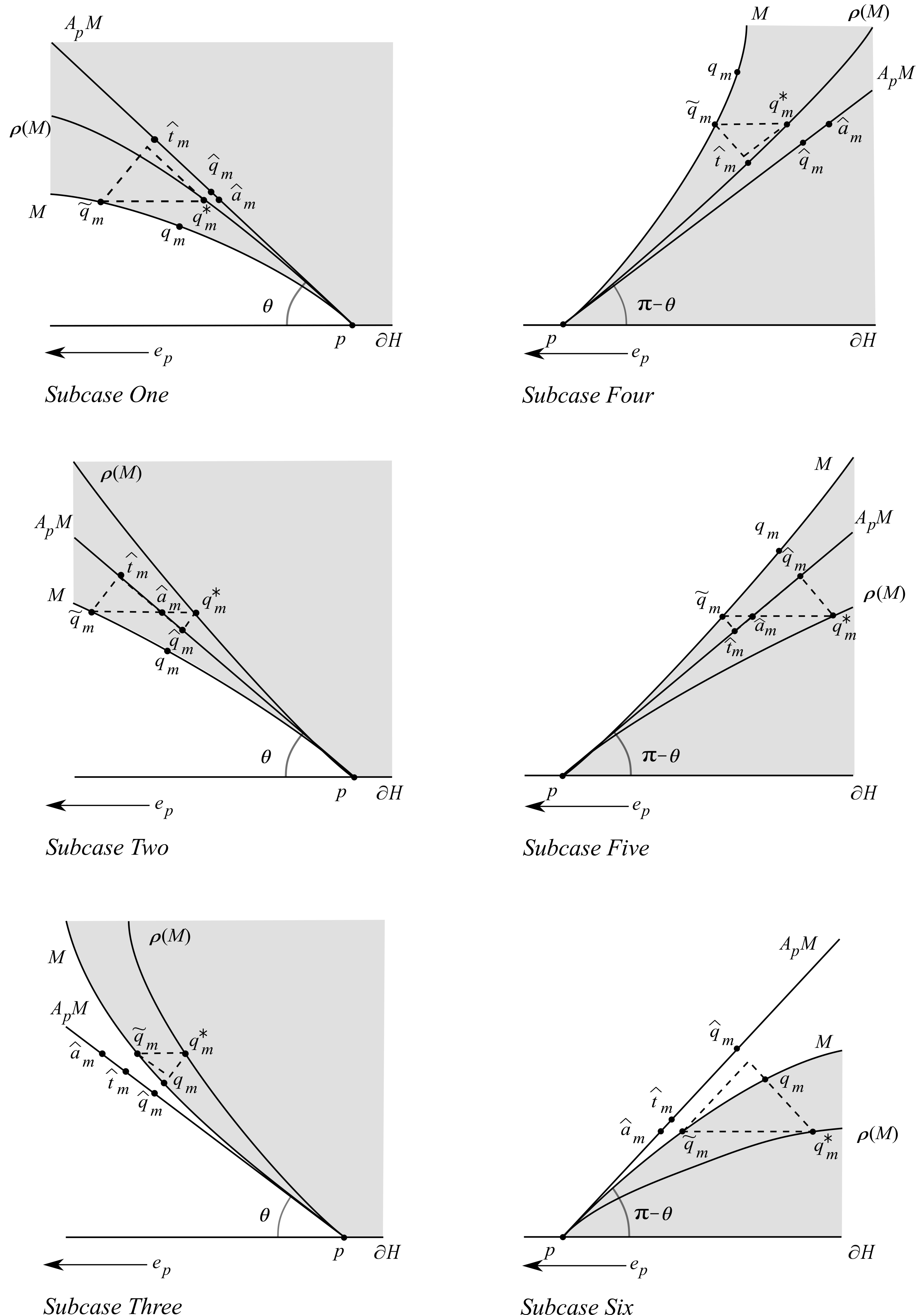}
 \caption{\small{The six subcases used in the proof of \eqref{numerator}. We consider cases depending on the sign of $\s$ and on the position of $M$ and $\rho(M)$ with respect to $A_pM$.} Here the shaded area represents the region inside of $E$.}\label{six cases}
\end{figure}

\noindent \textit{Subcase one}: Suppose $\s>0$ and $\hat{a}_m\cdot e_p \le q_m^*\cdot e_p\le \tilde{q}_m\cdot e_p$. 
Then 
\[
|B_m-b_m|\le \cot(\theta)f(B_m) \le C(\s)\tilde{f}(B_m)\,.
\]

\noindent \textit{Subcase two}: Suppose $\s>0$ and $q_m^*\cdot e_p\le \hat{a}_m\cdot e_p \le \tilde{q}_m\cdot e_p$. Then 
\[
|B_m-b_m|\le \cot(\theta)(f(B_m)+f^*(b_m)) \le C(\s)\tilde{f}(B_m)\,.
\]

\noindent \textit{Subcase three}: Suppose $\s>0$ and $q_m^*\cdot e_p\le\tilde{q}_m\cdot e_p \le  \hat{a}_m\cdot e_p$. Then 
\[
|B_m-b_m|\le \cot(\theta)(f^*(b_m)) \le C(\s)\tilde{f}(B_m)\,.
\]

\noindent \textit{Subcase four}: Suppose $\s<0$ and $\hat{a}_m\cdot e_p \le q_m^*\cdot e_p\le \tilde{q}_m\cdot e_p$. Then 
\[
|B_m-b_m|\le \cot(\pi-\theta)f(b_m) =-\cot(\theta)f(b_m) \le C(\s)\tilde{f}(B_m)\,.
\]

\noindent \textit{Subcase five}: Suppose $\s<0$ and $q_m^*\cdot e_p\le \hat{a}_m\cdot e_p \le \tilde{q}_m\cdot e_p$. Then 
\[
|B_m-b_m|\le -\cot(\theta)(f(b_m)+f^*(B_m)) \le C(\s)\tilde{f}(B_m)\,.
\]

\noindent \textit{Subcase six}: Suppose $\s<0$ and $q_m^*\cdot e_p\le\tilde{q}_m\cdot e_p \le  \hat{a}_m\cdot e_p$. Then 
\begin{equation*}
|B_m-b_m|\le -\cot(\theta)(f^*(B_m)) \le C(\s)\tilde{f}(B_m)\,.
\end{equation*}

\begin{figure}
\includegraphics[scale=.68]{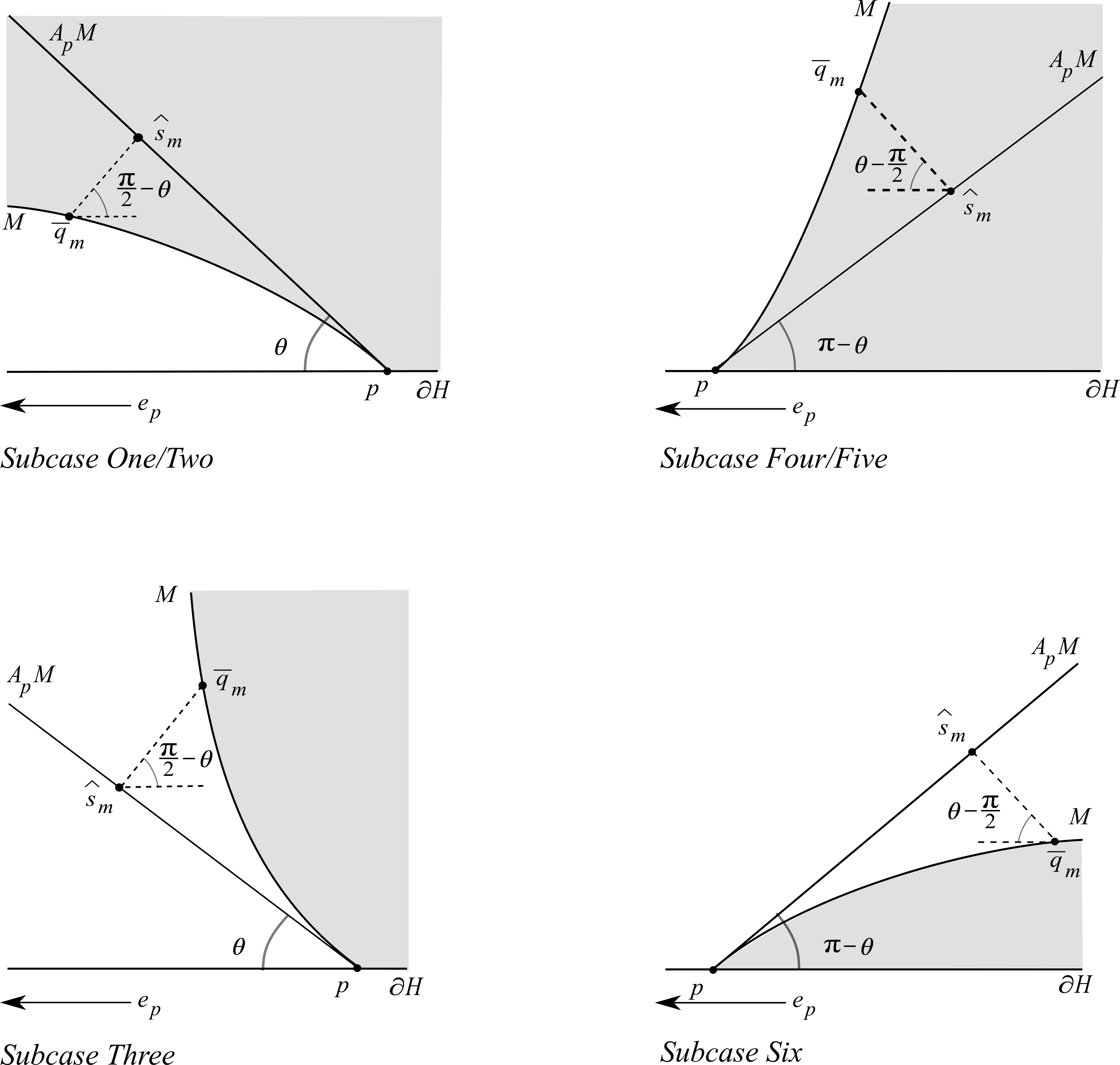}
 \caption{\small{The four subcases used to prove \eqref{denominator}. We can reduce the number of cases because we only need to consider the point $q_m$ or $\tilde{q}_m$ that minimizes the distance to $\pa H$. Here the shaded area represents the region inside of $E$.}}\label{six cases 2}
\end{figure}

Thus, by \eqref{f reg}, we have completed the proof of \eqref{numerator}. We now proceed with the proof of \eqref{denominator}. We will see that we can combine subcases one and two as well as subcases four and five, see Figure 4. Let $\bar{q}_m=q_m$ and $\hat{s}_m=\hat{q}_m$ if $q_m\cdot e_n=h_m$, and $\bar{q}_m=\tilde{q}_m$ and $\hat{s}_m=\hat{t}_m$  otherwise. Note that $b_m=|\hat{s}_m-p|.$  Exploiting that $h_m$ is approximately equal to $\hat{s}_m\cdot e_n$ will be essential in proving \eqref{denominator}.

\noindent \textit{Subcases one and two }: Suppose $\s>0$ and $\hat{a}_m\cdot e_p \le  \tilde{q}_m\cdot e_1$. 
Then, by \eqref{f reg}, for $m$ large enough
\[
h_m= b_m\sin\theta-f(b_m)\sin(\pi/2-\theta)=b_m\sin(\theta)-f(b_m)\cos\theta\ge b_m\sin(\theta)/2\,.
\]

\noindent \textit{Subcase three}: Suppose $\s>0$ and $   \tilde{q}_m\cdot e_p\le\hat{a}_m\cdot e_p$. Then 
\[
h_m= b_m\sin\theta+f(b_m)\sin(\pi/2-\theta)= b_m\sin(\theta)+f(b_m)\cos\theta\ge b_m\sin(\theta)\,.
\]

\noindent \textit{Subcases four and five}: Suppose $\s<0$ and $\hat{a}_m\cdot e_p \le  \tilde{q}_m\cdot e_p$. 
 Then 
\[
h_m= b_m\sin(\pi-\theta)+f(b_m)\sin(\pi/2-(\pi-\theta))= b_m\sin(\theta)+f(b_m)|\cos\theta|\ge b_m\sin(\theta)\,.
\]

\noindent \textit{Subcase six}: Suppose $\s<0$ and $\tilde{q}_m\cdot e_p\le\hat{a}_m\cdot e_p$.Then, by \eqref{f reg}, for $m$ large enough
\[
h_m= b_m\sin(\pi-\theta)-f(b_m)\sin(\pi/2-(\pi-\theta))= b_m\sin(\theta)-f(b_m)|\cos\theta|\ge b_m\sin(\theta)/2\,.
\]
Thus in each case we conclude that $h_m\ge b_m\sin(\theta)/2$. Because \eqref{real nume bound} implies
\[
 b_m=B_m-(B_m-b_m)\ge B_m-C(\s)B_m^{2+\a}\,,
\]we therefore have \eqref{denominator}, which completes the proof of step one.

\smallskip

\noindent {\it Step two}: In this step we will define an approximation $Q_m^\e$ of $Q_m$ and show that there is $\e(m)$ defined so that 
\begin{equation}\label{Step Two}
\liminf\limits_{m\rightarrow \infty}Q_m^{\e(m)}\ge\int_{E\setminus \rho(E)}\frac{2(n+s)(\l-x_1)(\l-p_1)}{\diam(E)^{n+s+2}}dx\,,
\end{equation}
and
\begin{equation}\label{e claim}
\lim_{m\rightarrow\infty}\e(m)= 0\,.
\end{equation}
First we define a new approximation of $\HH_E^{s}(q)$. 
For any $\e>0$, let 
\[
f_{\e}(z):=\left\{
\begin{array}{ll}
      \e^{n+s} & \mbox{if }z\in B_\e(0)\\
      |z|^{n+s} & z\in B_\e(0)^c\,,\\
\end{array} 
\right.
\] 
\[
\HH_E^{s,\e}(q):=\int_{\R^{n}}\frac{\tilde{\chi}_{E}(x)}{f_\e(|x-q|)}dx\,,
\]
and 
\[
Q_m^\e:=\HH_{\rho(E)}^{s,\e}(q_m^*)-\HH_E^{s,\e}(q_m)\,.
\]
By the continuity of $\HH_E^s$ we know that 
\[
\lim_{\e\rightarrow 0}\HH_E^{s,\e}(q)=\HH_E^s(q) \qquad\forall q\in M\cap H.
\]
We have
\begin{align}\label{Q e split}
Q_m^\e=&\int_{\R^n}\frac{\tilde{\chi}_{\rho(E)}(x)}{f_{\e}(|x-q_m^*|)}dx-\int_{\R^n}\frac{\tilde{\chi}_{E}(x)}{f_{\e}(|x-q_m|)}dx\nonumber\\
=&\int_{\R^n}\frac{\tilde{\chi}_{\rho(E)}(x)-\tilde{\chi}_{E}(x)}{f_{\e}(|x-q_m^*|)}dx +\int_{\R^n}\tilde{\chi}_{E}(x)\Big(\frac{1}{f_{\e}(|x-q_m^*|)}-\frac{1}{f_{\e}(|x-q_m|)}\Big)dx\,.
\end{align}
We handle the convergence of the two integrals separately. First let
\begin{equation}\label{I R def}
I^\e_m:=\int_{\R^n}\frac{\tilde{\chi}_{\rho(E)}(x)-\tilde{\chi}_{E}(x)}{f_{\e}(|x-q_m^*|)}dx\qquad R^{\e}_m:=\int_{\R^n}\tilde{\chi}_{E}(x)\Big(\frac{1}{f_{\e}(|x-q_m^*|)}-\frac{1}{f_{\e}(|x-q_m|)}\Big)dx\,.
\end{equation}
For any $\e>0$, we claim that, 
\begin{equation}\label{I claim}
\liminf_{m\rightarrow\infty} I^\e_m\ge \int_{(E\setminus \rho(E))\cap B_\e(p)^c}\frac{2(n+s)(\l-x_1)(\l-p_1)}{\diam(E)^{n+s+2}}dx\,.
\end{equation}
By symmetry across $\pi_\l$ we have
\begin{align}\label{case three to one}
I_\e=&2\int_{E\setminus \rho(E)}\frac{dx}{f_{\e}(|x-q_m^*|)}-2\int_{\rho(E)\setminus E}\frac{dx}{f_{\e}(|x-q_m^*|)}\,.\nonumber\\
=&2\int_{E\setminus \rho(E)}\Big(\frac{1}{f_{\e}(|x-q_m^*|)}-\frac{1}{f_{\e}(|\rho(x)-q_m^*|)} \Big)dx\nonumber\\
=&2\int_{(E\setminus \rho(E))\cap B_\e(q_m^*)^c}\Big(\frac{1}{f_{\e}(|x-q_m^*|)}-\frac{1}{f_{\e}(|\rho(x)-q_m^*|)} \Big)dx\nonumber\\
&+2\int_{(E\setminus \rho(E))\cap B_\e(q_m^*)}\Big(\frac{1}{f_{\e}(|x-q_m^*|)}-\frac{1}{f_{\e}(|\rho(x)-q_m^*|)} \Big)dx\,.
\end{align}
We consider each of the integrals in \eqref{case three to one} seperately. For the first integral we factor the integrand to get
\begin{equation}\label{factor I}
\frac{1}{f_{\e}(|x-q_m^*|)}-\frac{1}{f_{\e}(|\rho(x)-q_m^*|)} =\frac{1}{f_{\e}(|\rho(x)-q_m^*|)}\Big(\frac{f_{\e}(|\rho(x)-q_m^*|)}{f_{\e}(|x-q_m^*|)}-1 \Big)\,.
\end{equation}
We take $\e<(\l-p_1)/4$ so that $\rho(x)\in B_\e(q_m^*)^c$ for any $x\in E\setminus\rho(E)$. Then, following calculations in \eqref{case one calc two}, \eqref{case 1 mid}, and \eqref{case one pos} in case one, we know that if $x\in(E\setminus \rho(E)) \cap B_\e(q_m^*)^c $, then
\begin{equation}\label{^}
\frac{1}{f_{\e}(|\rho(x)-q_m^*|)}\Big(\frac{f_{\e}(|\rho(x)-q_m^*|)}{f_{\e}(|x-q_m^*|)}-1 \Big)\ge \frac{2(n+s)(\l-x_1)(\l-q_m^*\cdot e_1)}{\diam(E)^{n+s+2}}\,.
\end{equation}
Moreover if we let $\e<(\l-p_1)/4$ be fixed then as $m\rightarrow \infty$
\[
\chi_{B_\e(q_m^*)^c}\frac{2(n+s)(\l-x_1)(\l-q_m^*\cdot e_1)}{\diam(E)^{n+s+2}}\rightarrow \chi_{B_\e(p)^c}\frac{2(n+s)(\l-x_1)(\l-p_1)}{\diam(E)^{n+s+2}}\,.
\]
We know $(\l-x_1),(\l-q_m^*\cdot e_1)\le \diam(E)$, so 
\[
\Big|\chi_{B_\e(q_m^*)^c}\frac{2(n+s)(\l-x_1)(\l-q_m^*\cdot e_1)}{\diam(E)^{n+s+2}}\Big|\le \frac{2(n+s)}{\diam(E)^{n+s}}\,,
\]
for all $x\in E\setminus \rho(E),$ and therefore the dominated convergence theorem implies that
\begin{equation*}
\int_{(E\setminus \rho(E))\cap B_\e(q_m^*)^c}\frac{2(n+s)(\l-x_1)(\l-q_m^*\cdot e_1)}{\diam(E)^{n+s+2}}dx
\end{equation*}
\[\rightarrow\int_{(E\setminus \rho(E))\cap B_\e(p)^c}\frac{2(n+s)(\l-x_1)(\l-p_1)}{\diam(E)^{n+s+2}}dx\,,
\]
as $m\rightarrow\infty$. Hence for the first integral in \eqref{case three to one} we have
\[
\liminf_{m\rightarrow\infty}\int_{(E\setminus \rho(E))\cap B_\e(q_m^*)^c}\Big(\frac{1}{f_{\e}(|x-q_m^*|)}-\frac{1}{f_{\e}(|\rho(x)-q_m^*|)} \Big)dx\,.
\]
\begin{equation}\label{first integrand}
\ge\int_{(E\setminus \rho(E))\cap B_\e(p)^c}\frac{2(n+s)(\l-x_1)(\l-p_1)}{\diam(E)^{n+s+2}}dx\,.
\end{equation}
For the second integral in \eqref{case three to one}, having $\frac{1}{f_{\e}(|\rho(x)-q_m^*|)}< \e^{-(n+s)}$, for $x\in B_\e(q_m^*)$, implies that 
\[
\chi_{(E\setminus \rho(E))\cap B_\e(q_m^*)}\Big(\frac{1}{f_{\e}(|x-q_m^*|)}-\frac{1}{f_{\e}(|\rho(x)-q_m^*|)}\Big)
\]
is bounded above by $\e^{-(n+s)}$. So for fixed $\e$, the dominated convergence theorem implies that
\[
\int_{(E\setminus \rho(E))\cap B_\e(q_m^*)}\Big(\frac{1}{f_{\e}(|x-q_m^*|)}-\frac{1}{f_{\e}(|\rho(x)-q_m^*|)} \Big)dx
\]
\[
\rightarrow \int_{(E\setminus \rho(E))\cap B_\e(p)}\Big(\frac{1}{f_{\e}(|x-p|)}-\frac{1}{f_{\e}(|\rho(x)-p|)} \Big)dx\,,
\]
as $m\rightarrow \infty.$ But $f_{\e}(|\rho(x)-q_m^*|)=|\rho(x)-q_m^*|^{n+s}>2\e^{n+s}$. So for fixed $\e$ and $m$, and for $x\in (E\setminus \rho(E))\cap B_\e(q_m^*) $ we have
\begin{equation*}\begin{split}
\frac{1}{f_{\e}(|x-q_m^*|)}-\frac{1}{f_{\e}(|\rho(x)-q_m^*|)} 
&=\frac{1}{\e^{n+s}}-\frac{1}{|\rho(x)-q_m^*|^{n+s}}\\
&\ge\frac{1}{2\e^{n+s}}dx\ge 0\,.
\end{split}\end{equation*}
We conclude that, for large $m,$ the second integrand in \eqref{case three to one} is positive, therefore
\begin{equation}\label{second integral}
\liminf_{m\rightarrow\infty} \int_{(E\setminus \rho(E))\cap B_\e(q_m^*)}\Big(\frac{1}{f_{\e}(|x-q_m^*|)}-\frac{1}{f_{\e}(|\rho(x)-q_m^*|)}\Big)\ge 0\,,
\end{equation}
and, by \eqref{first integrand}, we have \eqref{I claim}. We now handle the convergence of $R^\e_m$ from \eqref{I R def}. We claim that there exists $\e(m)$ such that 
\begin{equation*}
\lim_{m\rightarrow\infty}\e(m)= 0
\end{equation*}
and 
\begin{equation}\label{R claim}
\lim_{m\rightarrow \infty} R^{\e(m)}_m=0\,.
\end{equation}
For $\e(m)$ to be chosen, let $X_{\e(m)}=B_{\e(m)}(q_m^*)\cap B_{\e(m)}(q_m)$. Note that
\[
\int_{X_{\e(m)}}\tilde{\chi}_{E}(x)\Big(\frac{1}{f_{\e(m)}(|x-q_m^*|)}-\frac{1}{f_{\e(m)}(|x-q_m|)}\Big)dx=0
\]
because $f_{\e(m)}(|x-q_m|)=f_{\e(m)}(|x-q_m^*|)={\e(m)}^{n+s}$ in $X_{\e(m)}$.
So
\begin{equation}\label{R def simple}
R^{\e(m)}_m=\int_{X_{\e(m)}^c}\tilde{\chi}_{E}(x)\Big(\frac{1}{f_{\e(m)}(|x-q_m^*|)}-\frac{1}{f_{\e(m)}(|x-q_m|)}\Big)dx\,.
\end{equation}
However we can rewrite
\[
\int_{X_{\e(m)}^c}\frac{\tilde{\chi}_{E}(x)\;dx}{f_{\e(m)}(|x-q_m|)}=\int_{B_{\e(m)}(q_m)^c}\frac{\tilde{\chi}_{E}(x)\;dx}{f_{\e(m)}(|x-q_m|)}+\int_{B_{\e(m)}(q_m)\setminus B_{\e(m)}(q_m^*)}\frac{\tilde{\chi}_{E}(x)\;dx}{f_{\e(m)}(|x-q_m|)}\,,
\]
and similarly 
\[
\int_{X_{\e(m)}^c}\frac{\tilde{\chi}_{E}(x)\;dx}{f_{\e(m)}(|x-q_m^*|)}=\int_{B_{\e(m)}(q_m^*)^c}\frac{\tilde{\chi}_{E}(x)\;dx}{f_{\e(m)}(|x-q_m^*|)}+\int_{B_{\e(m)}(q_m^*)\setminus B_{\e(m)}(q_m)}\frac{\tilde{\chi}_{E}(x)\;dx}{f_{\e(m)}(|x-q_m^*|)}\,,
\]
so, by \eqref{R def simple},
\begin{equation}\label{R bound}
\begin{split}
R^{\e(m)}_m\le &\Big|\int_{B_{\e(m)}(q_m)^c}\frac{\tilde{\chi}_{E}(x)\;dx}{f_{\e(m)}(|x-q_m|)}-\int_{B_{\e(m)}(q_m^*)^c}\frac{\tilde{\chi}_{E}(x)\;dx}{f_{\e(m)}(|x-q_m^*|)}\Big|\\
&+\Big|\int_{B_{\e(m)}(q_m)\setminus B_{\e(m)}(q_m^*)}\frac{\tilde{\chi}_{E}(x)\;dx}{f_{\e(m)}(|x-q_m|)}\Big|+\Big|\int_{B_{\e(m)}(q_m^*)\setminus B_{\e(m)}(q_m)}\frac{\tilde{\chi}_{E}(x)\;dx}{f_{\e(m)}(|x-q_m^*|)}\Big|\,.
\end{split}
\end{equation}
Note that there is a universal constant $C(n)$, dependent only on $n$, such that, for all $r>0$, we know
\[
|B_r(w)\Delta B_r(0)|\le C(n)r^{n-1}|w| \qquad \forall w\in\R^n\,.
\]
Therefore, we have 
\begin{equation}\label{3 main 1}
\int_{B_{\e(m)}(q_m)\setminus B_{\e(m)}(q_m^*)}\frac{\tilde{\chi}_{E}(x)\;dx}{f_{\e(m)}(|x-q_m|)}\le\frac{|B_{\e(m)}(q_m^*)\Delta B_{\e(m)}(q_m)|}{2{\e(m)}^{n+s}}\le\frac{C(n)|q_m-q_m^*|}{{\e(m)}^{s+1}}\,,
\end{equation}
and
\begin{equation}\label{3 main 2}
\int_{B_{\e(m)}(q_m^*)\setminus B_{\e(m)}(q_m)}\frac{\tilde{\chi}_{E}(x)\; dx}{f_{\e(m)}(|x-q_m^*|)}\le \frac{|B_{\e(m)}(q_m^*)\Delta B_\e(q_m)|}{2{\e(m)}^{n+s}}\le\frac{C(n)|q_m-q_m^*|}{{\e(m)}^{s+1}}\,.
\end{equation}
Also, by applying \cite[Remark 3.25]{AmbFusPal} with $u=\chi_E$, 
\[
\Big|\int_{B_{\e(m)}(q_m)^c}\frac{\tilde{\chi}_{E}(x)}{f_{\e(m)}(|x-q_m|)}dx-\int_{B_{\e(m)}(q_m^*)^c}\frac{\tilde{\chi}_{E}(x)}{f_{\e(m)}(|x-q_m^*|)}dx\Big|\quad
\]
\begin{align}\label{3 main 3}
&=\Big|\int_{B_{\e(m)}(0)^c}\Big(\frac{\tilde{\chi}_{E}(x+q_m)-\tilde{\chi}_{E}(x+q_m^*)}{|x|^{n+s}}\Big)dx\Big|\nonumber\\
&\le\int_{B_{\e(m)}(0)^c}\frac{\chi_{(E+q_m)\Delta(E+q_m^*)}}{|x|^{n+s}}dx\nonumber\\
&\le \frac{P(E)|q_m-q_m^*|}{{\e(m)}^{n+s}}\,.
\end{align}
We choose $\b>0$, and set $\e(m)=|q_m-q_m^*|^{1/(n+s+\b)}$. Note that \eqref{e claim} holds because 
\[
|q_m-q_m^*|=O\big((1/m)^{2+\a})\,.
\]
Moreover, for this choice of $\e(m)$, we also have that \eqref{3 main 1}, \eqref{3 main 2}, and \eqref{3 main 3} converge to $0$ as $m\rightarrow\infty,$ and therefore with \eqref{R bound}, we have \eqref{R claim}. By \eqref{e claim}, by \eqref{I claim}, and because of our choice of $\e(m)$, we have   
\begin{equation}\label{I claim real}
\liminf_{m\rightarrow\infty} I^{\e(m)}_m\ge \int_{E\setminus \rho(E)}\frac{2(n+s)(\l-x_1)(\l-p_1)}{\diam(E)^{n+s+2}}dx\,.
\end{equation}
Hence, by \eqref{Q e split}, together \eqref{I claim real} and \eqref{R claim} imply \eqref{Step Two}.

\smallskip

\noindent {\it Step three}: 
We show that
\begin{equation}\label{switch order}
\liminf\limits_{m\rightarrow \infty}Q_m^{\e(m)}=\liminf\limits_{m\rightarrow \infty}\HH_{\rho(E)}^{s,\e(m)}(q_m^*)-\HH_E^{s,\e(m)}(q_m)=\lim\limits_{m\rightarrow \infty} \HH_{\rho(E)}^s(q_m^*)-\HH_E^{s}(q_m)=\lim\limits_{m\rightarrow \infty}Q_m\,,
\end{equation}
to get \eqref{int diff} in case three. To see this convergence we show that 
\begin{equation}\label{conv 331}
\lim\limits_{m\rightarrow \infty} |\HH_E^{s,\e(m)}(q_m)-\HH_E^s(q_m)|=0\,,
\end{equation}
and equivalently that
\begin{equation}\label{conv 332}
\lim\limits_{m\rightarrow \infty} |\HH_{\rho(E)}^{s,\e(m)}(q_m^*)-\HH_{\rho(E)}^s(q_m^*)|=0\,.
\end{equation}
Here we use the properties of the graphs locally defining $M$ and $\rho(M)$. By definition,
\begin{equation*}
\big|\HH_{E}^{s,\e(m)}(q_m)-\HH_E^s(q_m)\big|=\Big|\int_{B_{\e(m)}(q_m)}\tilde{\chi}_E(x)\Big[\frac{1}{|x-q_m|^{n+s}}-\frac{1}{\e(m)^{n+s}}\Big]dx\Big|\,.
\end{equation*}
Because $\e(m)\rightarrow 0$ as $m\rightarrow\infty$, we know that $\e(m)<\eta$, for $m$ large enough and for $\eta$ as in the definition of $P_{\eta,\g}(q)$ in \eqref{parabola}. So, for $m$ large enough, we have
\begin{equation*}\begin{split}
\big|\HH_E^{s,\e(m)}(q_m)-\HH_E^s(q_m)\big|&=\Big|\int_{ P_{\e(m),\g}(q_m)}\tilde{\chi}_E(x)\Big(\frac{1}{|x-q_m|^{n+s}}-\frac{1}{\e(m)^{n+s}}\Big)dx\Big|\\
&\le\int_{ P_{\e(m),\g}(q_m)}\frac{1}{|x-q_m|^{n+s}}dx\,,
\end{split}\end{equation*}
which goes to 0 as $m\rightarrow \infty.$ Therefore we have \eqref{conv 331} and \eqref{conv 332}, and thus \eqref{switch order} as claimed.
Hence, by \eqref{Step Two} from step two we have
\[
\lim\limits_{m\rightarrow \infty} Q_m\ge\int_{E\setminus \rho(E)}\frac{2(n+s)(\l-x_1)(\l-p_1)}{\diam(E)^{n+s+2}}dx\,,
\]
which combined with \eqref{*} from step one, implies
\[
\frac{5|\l-p_1|\de_s(E)}{2\diam(E)^{s+1}}\ge\int_{E\setminus \rho(E)}\frac{2(n+s)(\l-x_1)(\l-p_1)}{\diam(E)^{n+s+2}}dx\,.
\]
Therefore, since $\l-p_1>0,$ we have \eqref{int diff} in case three.

\medskip

\noindent\textit{Case four}: In this case we assume that $p\in \pi_\l\cap\bd(M).$ As in case three we use the local graphicality of $M$ and $\rho(M)$ near $p$ to find a sequence of points to approximate the work done in case two. We want a sequence of points approaching $p$ in $M\cap \pi_\l=\rho(M)\cap \pi_\l.$ Let $\hat{q}_m:=p-\nu_{co}^M(p)/m,$ where $m\in \N$ and  $\nu_{co}^M(p)$ is the conormal vector for $p$ with respect to $M.$ Set $q_m=v(\hat{q}_m)$ for $v$ as in case three. Let $u_{\e,q}(x):=\var_\e(|x-q|)$ and let
\[
\HH_E^{s,\e}(q):={\rm p.v.}\int_{\R^n}\tilde{\chi}_{E}(x)\var_\e(|x-q|)dx\,,
\] 
as in case two. For any $\t_m\in T_{q_m}M$ such that $\t_m\cdot e_n=0$, by \cite[Lemma 2.1]{ciraolofigallimagginovaga}, we know 
\[
\nabla \HH_E^{s,\e}(q_m)\cdot \t_m= \int_{\R^n}\tilde{\chi}_E(x)\nabla u_{\e,q_m}(x)\cdot \t_m\;dx\,.
\]
Let $P_m$ be the hyperplane through $q_m$ which is perpendicular to $\t_m.$ Set $E_m^*$ to be the reflection of $E$ across $P_m.$ Then, because $\nabla u_{\e,q_m}(x)$ is odd with respect to $P_m,$ and because $\t_m\cdot e_n=0$, we have
\[
\nabla \HH_E^{s,\e}(q_m)\cdot \t_m=- 2\int_{E\setminus E_m^*}\nabla u_{\e,q_m}(x)\cdot \t_m \;dx\,.
\]
But because $\nabla u_{\e,q_m}(x)\cdot \t_m=\var'_\e(|x-q_m|)\frac{(x-q_m)\cdot \t_m}{|x-q_m|}$ we know
\begin{equation*}\begin{split}
\nabla \HH_E^{s,\e}(q_m)\cdot \t_m&=- 2\int_{E\setminus E_m^*}\var'_\e(|x-q_m|)\frac{(x-q_m)\cdot \t_m}{|x-q_m|}dx\\
&=- 2\int_{\R_n}\chi_{E\setminus E_m^*}(x)\var'_\e(|x-q_m|)\frac{(x-q_m)\cdot \t_m}{|x-q_m|}dx\,.
\end{split}\end{equation*}
Therefore, by the monotone convergence of $|\vphi'_\e|$ in $\e$
\begin{equation}\label{tangent limit}
\lim_{\e\rightarrow 0}\nabla\HH_E^{s,\e}(q_m)\cdot \t_m=-2(n+s)\int_{\R^n}\chi_{E\setminus E^*_m}(x)\frac{(q_m-x)\cdot \t_m}{|x-q_m|^{n+s+2}}dx\,.
\end{equation}
Next, set
\[
\t_m:=e_1-(w_m\cdot e_1)\frac{w_m}{||w_m||}
\]
with
\[
w_m:=\nu_{E}(q_m)-(\nu_{E}(q_m)\cdot e_n)e_n\,.
\]
Then $\t_m\rightarrow e_1$ and $\chi_{E\setminus E_m^*}\rightarrow \chi_{E\setminus \rho(E)}$ pointwise in $m$ because $\nu_E(p)\cdot e_1=0$ and $\nu_E(q_m)\rightarrow \nu_E(p)$ as $m\rightarrow\infty$ imply that $w_m\cdot e_1\rightarrow 0$ as $m\rightarrow\infty$. So
\[
\lim\limits_{m\rightarrow\infty}\chi_{E\setminus E^*_m}(x)\frac{(q_m-x)\cdot \t_m}{|x-q_m|^{n+s+2}}=\chi_{E\setminus \rho(E)}(x)\frac{(p-x)\cdot e_1}{|x-p|^{n+s+2}}
\]
pointwise. By definition, $\t_m\cdot e_n=0$ and $\t_m\in T_{q_m}M$. Therefore, by applying Fatou's lemma to \eqref{tangent limit}, we have
\[
\liminf\limits_{m\rightarrow\infty}\lim_{\e\rightarrow 0}-\nabla\HH_E^{s,\e}(q_m)\cdot \t_m\ge\int_{\R_n}\chi_{E\setminus \rho(E)}(x)\frac{(p-x)\cdot e_1}{|x-p|^{n+s+2}}dx=\int_{E\setminus \rho(E)}\frac{(p-x)\cdot e_1}{|x-p|^{n+s+2}}dx\,.
\]
Thus, because $\frac{(p-x)\cdot e_1}{|x-p|^{n+s+2}}$ is positive by construction and 
\[
\de_s(E)\ge\diam(E)^{s+1}|\nabla \HH_E^{s,\e}(q_m)\cdot \t_m|\,,
\]
 for all $m$ large, we have \eqref{int diff} in case four.

\smallskip

\noindent\textit{Conclusion to part (a)}:
We have concluded in every case that 
\begin{equation*}
\int_{E\Delta \rho(E)}\dist(x,\pi_\l)dx\le \frac{5}{2(n+s)}\diam(E)^{n+1}\de_s(E)\,.
\end{equation*}
Thus, by Chebyshev's inequality, for any $\b>0$,
\[
\big|\{x\in E\Delta\rho(E):\dist(x,\pi_\l)\ge \b\}\big|\le\frac{1}{\b}\frac{5}{2(n+s)}\diam(E)^{n+1}\de_s(E)\,.
\]
So because 
\[
\big|\{x\in E\Delta\rho(E):\dist(x,\pi_\l)\le \b\}\big|\le 2\b\diam(E)^{n-1}\,,
\]
by setting 
\[
\b=\sqrt{\frac{5}{2(n+s)}}\diam(E)\sqrt{\de_s(E)}\,,
\]
we have
\begin{equation*}
|E\Delta\rho (E)|\le C_1\diam(E)^{n}\sqrt{\de_s(E)}\,,
\end{equation*}
for 
\begin{equation}\label{constant one}
C_1=3\sqrt{\frac{5}{2(n+s)}}\,.
\end{equation}
\end{proof}

\section{Improved Understanding of the Geometry}
We will use the following lemma to prove Theorem \ref{thm axially symmetric}\ref{improved thm}. The proof closely follows the proof of \cite[Lemma 4.1]{ciraolofigallimagginovaga}.

\begin{lemma}\label{l bound 1} Let $E$ be as in assumption (h1) with 
\begin{equation}\label{delta bound}
\sqrt{\de_s(E)}\le \frac{1}{3} \min\Big\{\frac{1}{2},\frac{1}{n-1}\Big\}\sqrt{\frac{2(n+s)}{5}}\frac{|E|}{\diam(E)^{n}}\,.
\end{equation}
Suppose that the critical planes with respect to the coordinate directions, coincide with $\{x_i=0\}$ for all $i=1,...,n-1$. If $e\in S^{n-1}\cap\pa H$, and $\l_e$ is the critical value associated with $e$ in \eqref{lambda}, then 
\begin{equation*}\label{lambda bound}
|\l_e| \le C\frac{\diam(E)^{n+1}}{|E|}\sqrt{\de_s(E)}
\end{equation*} 
holds for some constant $C=C(n).$
\end{lemma}
\begin{proof}
We begin by letting $F^0:=\{(-\hat{x},x_n):x\in F\}$ for $F\subset \R^n$. Set 
\[
C_1^*:=C_1\diam(E)^{n}\,.
\]
with $C_1$ as defined in \eqref{constant one} from the proof of Theorem \ref{thm axially symmetric}(a). Theorem~\ref{thm axially symmetric}(a) implies 
\begin{equation}\label{combined thm}
|E\Delta E^0|\le (n-1)C_1^*\sqrt{\de_s(E)}\,,
\end{equation}
because $E^0$ can be obtained from $E$ by symmetrizing with respect to the $n-1$ hyperplanes $\{x_i=0\}$ for $i=1,..,n-1$.
So by \eqref{delta bound} we have $|E\Delta E^0|\le |E|.$ Without loss of generality we assume $\l_e>0$. Note that 
\begin{equation}\label{big L bound}
\L_e:=\sup\{x\cdot e:x\in E\}\le \diam(E)\,.
\end{equation}
If not, we would have $x\cdot e\ge 0$ for all $x\in E,$ which would imply that $|E\Delta E^0|=2|E|$. Moreover, by Theorem~\ref{thm axially symmetric}(a) we also have  
\begin{equation}\label{refl bound}
|E\Delta \rho(E)|\le C_1^*\sqrt{\de_e(E)}\,,
\end{equation}
with $\rho$ defined to be the function that reflects $E$ across the critical hyperplane $\pi_e$. 
So
\begin{equation}\label{most mass}
|E\cap\{x\cdot e>\l_e\}|\ge\frac{1}{2}(|E|-C_1^*\sqrt{\de_s(E)})\,.
\end{equation}
Then, by \eqref{combined thm} and \eqref{most mass}, we have
\[
|E\cap \{x\cdot e>\l_e\}^0|=|E^0\cap \{x\cdot e>\l_e\}|\ge |E\cap\{x\cdot e>\l_e\}|-|E\Delta E^0|\ge\frac{|E|}{2}-\Big(n-\frac{1}2\Big)C_1^*\sqrt{\de_s(E)}\,,
\]
so, by applying \eqref{most mass} again, we know
\begin{align}\label{small strip}
|\{x\in E:-\l_e\le x\cdot e = \l_e\}|\le&\, |\{x\in E:x\cdot e<0\}|-|E\cap \{x\cdot e>\l_e\}^0|\nonumber \\
&+ |\{x\in E:x\cdot e>0\}| -|E\cap\{x\cdot e>\l_e\}|  \nonumber\\
\le&\, nC_1^*\sqrt{\de_s(E)}\,.
\end{align}
We have shown that $E$ has small volume in the strip $\{|x\cdot e|<\l_e\}.$ We continue by quantifying the volume of $E$ in parallel strips of the same width. The set $\{\l_e\le x\cdot e\le 3\l_e\}$ is mapped into $\{|x\cdot e|<\l_e\}$ by reflection with respect to the critical hyperplane, so by \eqref{refl bound} and \eqref{small strip} we have
\begin{equation*}\begin{split}
|\{x\in E:\l_e\le x\cdot e \le 3\l_e\}|&=|\{x\in \rho(E):|x\cdot e| \le \l_e\}|\\
&\le|\{x\in E:|x\cdot e| \le \l_e\}|+|E\Delta \rho(E)|\le (n+1)C_1^*\sqrt{\de_s(E)}\,.\\
\end{split}\end{equation*}
Now let
\[
m_k:=|\{x\in E:(2k-1)\l_e\le x\cdot e\le (2k+1)\l_e\}|, \qquad k\ge 1\,.
\]
By the moving planes procedure, if $\l_e\le \mu'\le\mu$, then $E\cap \pi_{\mu}\subset E\cap\pi_{\mu'}$, where each set is seen as a subset in $\R^{n-1}.$ So because $\H^{n-1}(E\cap\pi_\mu)$ is decreasing in $\mu$, for $\mu\in(\l_e,\L_e)$, it follows that $m_k$ is decreasing in $k.$ Therefore
\begin{equation}\label{m gets smaller}
m_k\le m_1\le (n+1)C_1^*\sqrt{\de_s(E)}\qquad k\ge 1\,.
\end{equation}
Let $k_0$ be the smallest natural number such that $(2k_0+1)\l_e\ge \L_e,$ which implies $(2k_0-1)\l_e\le \L_e.$  Because  $E\subset\{x\cdot e\le \L_e\}$, by \eqref{m gets smaller} we see that 
\begin{equation*}
|E\cap\{x\cdot e>\l_e\}|=|E\cap\{\l_e\le x\cdot e\le \L_e\}|=\sum_{k=1}^{k_0}m_k\le \frac{1}2\Big(\frac{\L_e}{\l_e}+1\Big)(n+1)C_1^*\sqrt{\de_s(E)}\,.
\end{equation*}
Then \eqref{big L bound} implies
\begin{equation}\label{almost}
|E\cap\{x\cdot e>\l_e\}|\l_e\le (n+1)C_1^*\diam(E)\sqrt{\de_s(E)}\,.
\end{equation}
Lastly, by (\ref{delta bound}) and (\ref{most mass}), we know $|E\cap\{x\cdot e>\l_e\}|\ge|E|/4$, which combined with  \eqref{almost} and the definition of $C_1$ implies 
\begin{equation}\label{real lambda bound}
\l_e\le \frac{(n+1)C_1^*\diam(E)\sqrt{\de_s(E)}}{|E\cap\{x\cdot e>\l_e\}|}\le 4(n+1)C_1\frac{\diam(E)^{n+1}}{|E|}\sqrt{\de_s(E)}\,,
\end{equation}
as required.\end{proof}

\begin{proof}[Proof of Theorem 1.1 (b)] Let 
\[
\de_0=\frac{1}{3} \min\Big\{\frac{1}{2},\frac{1}{n-1}\Big\}\sqrt{\frac{2(n+s)}{5}}\frac{|E|}{\diam(E)^{n}}.
\] We assume that $\de(E)\le \de_0$, so Lemma \ref{l bound 1} applies. Up to translation, we may also assume that the critical planes with respect to the coordinate directions $e_i$ coincide with $\{x_i=0\}$ for every $i=1,...,n-1.$
Let $E_h$ be the cross section of $E$ parallel to $\pa H$ at height $h.$ For each height $h$ let 
\[
r_h=\inf\limits_{x\in \pa E_h}|x-he_n|\qquad\mbox{ and }\qquad R_h=\sup\limits_{x\in \pa E_h}|x-he_n|\,.
\]
Choose $x_h,y_h$ such that $|x_h-h e_n|=r_h$ and $|y_h-he_n|=R_h$. Without loss of generality we may assume that $x_h\ne y_h$, otherwise $R_h-r_h=0$. Let 
\[
e_h:=\frac{y_h-x_h}{|y_h-x_h|}\,,
\]
let $\l_h=\l_{e_h}$ be the critical value for $e_h,$ and let $\pi_{h}=\pi_{\l_{e_h}}$ denote the critical hyperplane. 
 Note that $y_h$ is closer to $\pi_h$ than $x_h$ in $\{x_n=h\}$, that is, 
\begin{equation}\label{dist plane}
\dist(x_h,\pi_{h})\ge\dist(y_h,\pi_{h})\,.
\end{equation}
Indeed, by the moving planes method, the critical position can be reached at most when $\rho(y_h)$, the reflection of $y_h$ with respect to $\pi_h$, is $x_h$. In this case we have equality in \eqref{dist plane}, but otherwise we have strict inequality. So by \eqref{dist plane}, and because $e_h$ is parallel to $y_h-x_h,$ we have
\begin{equation}\label{rs diff}\begin{split}
R_h-r_h&=|y_h-he_n|-|x_h-he_n|\le |(y_h-h e_n)-(x_h-h e_n)|\\
&=(\dist(y,\pi_{h})+\l_{h})-(\dist(x,\pi_{h})-\l_{h})\\
&\le2|\l_{h}|\,.
\end{split}\end{equation}
Set $C_2=4(n+1)C_1$, with $C_1$ from \eqref{constant one}. Combining \eqref{real lambda bound} with \eqref{rs diff}, we get
\begin{equation*}\label{radii dist}
R_h-r_h\le 2C_2\frac{\diam(E)^{n+1}}{|E|}\sqrt{\de_s(E)}\,,
\end{equation*}
or equivalently 
\begin{equation*}\label{radii dist}
\frac{R_h-r_h}{\diam(E)}\le 2C_2\frac{\diam(E)^{n}}{|E|}\sqrt{\de_s(E)}\,,
\end{equation*}
which is \eqref{R diff bound}.

Note that  \eqref{R diff bound} just implies that the boundary of $E_h$ is contained in an annulus with radii $r_h$ and $R_h$. In fact, if $\diam(E_h)$ is small enough then $E_h$ could be contained in the annulus as well, and may not contain the inner ball $D_{h}:=B_{r_h}(he_n)\cap\{x_n=h\}$. However, we will show that if 
\begin{equation}\label{slice di bound}
\frac{\diam(E_h)}{\diam(E)}>  6C_2\frac{\diam(E)^{n}}{|E|}\sqrt{\de_s(E)}\,,
\end{equation}
then $D_h\subset E_h$. Suppose that $D_{h}$ is not contained in $E_h$. 
By applying the moving planes argument in any direction $e$ such that $e\cdot e_n=0$ and $\{x\in E_h:e\cdot x/|x|=\pm1\}\ne\emptyset$, we know that 
\[
|\l_e|\ge r_h.
\]
Together with \eqref{real lambda bound} this implies  
\begin{equation*}\label{small di r bound}
r_h\le C_2\frac{\diam(E)^{n+1}}{|E|}\sqrt{\de_s(E)}\,,
\end{equation*}
so 
\begin{equation*}
\diam(E_h)\le 2R_h\le 2\Big(2C_2\frac{\diam(E)^{n+1}}{|E|}\sqrt{\de_s(E)}+r_h\label{small enough di}\Big)\le  6C_2\frac{\diam(E)^{n+1}}{|E|}\sqrt{\de_s(E)}\,,
\end{equation*}
which is equivalent to \eqref{slice di bound}.
In particular, combining this inequality with $\diam(E_h)\ge R_h-r_h$ gives
\[
R_h\le \diam(E_h)+r_h\le 7C_2\frac{\diam(E)^{n+1}}{|E|}\sqrt{\de_s(E)}\,,
\]
which concludes the proof of the theorem. 
\end{proof}
%

\bibliography{references}
\bibliographystyle{is-alpha}

\end{document}